\documentclass{amsart}
\usepackage{geometry}                
\usepackage{graphicx}
\usepackage{amssymb}
\usepackage{epstopdf,amsmath}
\usepackage{enumitem}

\newtheorem{prop}{Proposition}[section]

\newtheorem{lemma}[prop]{Lemma}
\newtheorem*{lemma*}{Lemma}

\newtheorem{theorem}[prop]{Theorem}
\newtheorem{conj}[prop]{Conjecture}

\newtheorem*{prop*}{Proposition}
\newtheorem*{theorem*}{Theorem}

\newcommand{\FF}{\mathbb{F}}
\newcommand{\CC}{\mathbb{C}}

\newcommand{\eps}{\epsilon}

\newcommand{\QQ}{\mathbb{Q}}
\newcommand{\ZZ}{\mathbb{Z}}
\newcommand{\RR}{\mathbb{R}}

\newcommand{\NN}{\mathbb{N}}

\newcommand{\house}{\textrm{House}}
\newcommand{\heur}{\approx}

\newcommand{\Average}{\textrm{Average}}
\newcommand{\Volume}{\textrm{Volume}}
\newcommand{\Cov}{\textrm{Cov}}
\newcommand{\Rect}{\textrm{Rect}}
\newcommand{\Ann}{\textrm{Ann}}
\newcommand{\Cover}{\textrm{Cover}}
\newcommand{\PGL}{\textrm{PGL}}
\newcommand{\trdeg}{\textrm{Transc Deg}}

\title{Perspectives on the unit distance problem}

\author{Larry Guth}

\begin{document}

\maketitle

\begin{abstract}
This is a survey on an old open problem in combinatorics called the unit distance problem, and the field of mathematics around it, called incidence geometry.   What do we know about the problem?  Why is it difficult?  How does it connect with other parts of math?
\end{abstract}

\section{Introduction}

The unit distance problem is an old problem in combinatorial geometry,  raised by Paul Erdos in the 1940s.  The problem is simple to state but we remain far from understanding it.

Recently,  a paper by Open AI \cite{OAI} gave new examples for this problem for the first time since Erdos's work in the 1940s,  going against the expectations of many people.   Following this example,  mathematicians found related counterexamples to several other longstanding problems in combinatorial geometry and number theory,  including \cite{BSSZ} and \cite{Po1}.  

This is a survey paper about the unit distance problem and the field around it.   I will describe why I think the problem is interesting.  We will discuss why the problem is difficult.  We will discuss how problems of this type relate to other parts of math.  We will describe the new examples,  and we will discuss how they fit into the field.

Here is the unit distance problem.   Suppose that $P$ is a finite set of points in the plane $\RR^2$.   We let $U(P)$ be the number of pairs of $P$ at unit distance:

$$ U(P) := \# \{ p_1, p_2 \in P: |p_1 - p_2| = 1 \}. $$

\noindent Let $U_{max}(n) = \max_{|P| = n} U(P)$.   The problem is to understand the asymptotics of $U_{max}(n)$.

In the 1940s,  Erdos proved that $U_{max}(n) \le C n^{3/2}$,  and he conjectured that $U_{max}(n) \le C_\epsilon n^{1 + \epsilon}$ for any $\epsilon > 0$.   In the 1980s, using an interesting argument based on topology,  Spencer,  Szemeredi,  and Trotter \cite{SST} proved that $U_{max}(n) \le C n^{4/3}$.   This bound has not been improved.   The recent Open AI paper disproved Erdos's conjecture.  They proved

\begin{theorem} \label{counterunit} (\cite{OAI}) There is some $\epsilon_0 > 0$ and $c > 0$ so that $U_{max}(n) \ge c n^{1 + \epsilon_0}$ for a sequence of $n$ tending to infinity.  
\end{theorem}

\noindent A followup paper by Sawin \cite{Sa} refined the argument to give an explicit $\epsilon_0 = 0.014...$

 The unit distance problem was among the first problems in a field of combinatorial geometry called incidence geometry.  Let us now introduce incidence geometry more broadly.

Imagine a hypothetical set $P$ of $n$ points with many unit distances.   Imagine drawing a unit circle around each point of $P$.   Those unit circles will have to hit the points of $P$ many times.   Suppose that $P$ is a set of points and $\Gamma$ is set of curves (or more generally a set of surfaces or subvarieties or ...).   We say a point $p \in P$ is incident to a curve $\gamma \in \Gamma$ if $p \in \gamma$,  and we write $I(P,\Gamma)$ for the number of incidences:

$$ I(P, \Gamma) = \# \{ p \in P,  \gamma \in \Gamma: p \in \gamma \}. $$

\noindent The unit distance problem boils down to estimating $I(P, \Gamma)$ where $\Gamma$ is a set of unit circles.

In general,  let $C$ denote any class of subvarieties of $\RR^n$.  For instance,  $C$ could be the class of straight lines in the plane,  unit circles in the plane,  ellipses in the plane,  2-spheres in $\RR^3$,  ...   We define

$$ I_{max,  C}(m,n) = \max_{ |P| = m,  |\Gamma| = n }  I(P, \Gamma) . $$

\noindent (In the $\max$,  we take $\Gamma$ to be a set of $n$ curves/surfaces from $C$.)   We first introduced the problem using curves/surfaces in $\RR^n$,  but we can also replace $\RR$ with other fields,  including $\CC$ or finite fields $\FF_q$.   Considering the many possible choices of $C$ and also the choice of field,  we get dozens of different problems.   Almost all of these problems are wide open. 

There is one important example where $I_{max,C}(m,n)$ is well understood: the case of straight lines in the plane $\RR^2$.

\begin{theorem} \label{thmst} (Szemeredi-Trotter,  \cite{ST},  1980s) If $C$ is the class of straight lines in $\RR^2$,  then $I_{max,C}(m,n) \sim m^{2/3} n^{2/3} + m + n$.
\end{theorem}

The examples with  $|I(P,\Gamma)| \sim m^{2/3} n^{2/3}$ are based on simple number theory.   For instance,  we can arrange the points $P$ in an $m \times m$ integer grid and choose lines with rational slopes (of small height) which each intersect the grid at  many points.   Most of the interesting examples in the field are based on number theory.   We will explore the role of number theory more as we go along.   The recent new examples for the unit distance problem are also based on number theory,  building on previous number theoretic examples but with a significant new twist involving high degree number fields.

The upper bounds in the Szemeredi-Trotter theorem are related to topology.   There are several interesting proofs,  but they all use topology in some way.   It is rather remarkable that the upper bounds,  proven using topology,  match the lower bounds,  proven using number theory.    (On the other hand,  for almost every other choice of $C$,  the asymptotic behavior of $I_{max,C}$ is unknown.)   The Spencer-Szemeredi-Trotter upper bound $U_{max}(n) \le C n^{4/3}$ is closely related to the Szemeredi-Trotter theorem.

Before we go on,  I would like to make a little personal reflection about this style of question.   In the 1930s, 40s, and 50s,  combinatorics was growing into its own field.   Erdos and others working in the field posed many questions.   I think they tried to pose questions with very simple words and leading in new directions.  Some of these questions led to new subfields and started mathematical
explorations still unfolding today. I think that asking questions in simple words like this is quite remarkable. If we compare with other fields of science, imagine going to the house of an engineer
and asking what they are working on, and imagine that they take out their child’s erector set and put a few pieces together and ask a question about it, and imagine this question is actually
a serious direction of research in engineering. I vividly remember learning about the Szemeredi-Trotter theorem and the unit distance problem from my friend Matt Kahle when I was a postdoc at Stanford. I remember how refreshing it felt to see interesting recent work and open questions about lines or circles in the plane. I have heard that Erdos tried to ask questions in simple words
partly to help engage young people and make the community more vibrant.

When the unit distance problem first appeared,  I don't think anyone had a good guess how difficult it would be.   I think it seemed plausible that it might have a one page solution.   But the problem has turned out to be really difficult,  and we will spend some time reflecting on why.
When the unit distance problem first appeared,  it was also not very clear how much it connected to other parts of math.   As the field of incidence geometry grew,  people found a number of significant connections to other areas of math.   We will review some of these in the next section.    

Here is an outline of the rest of the paper.  In Section \ref{secconnections},  we will discuss some of the connections between incidence geometry and other parts of math,  which gives one point of view about why incidence geometry is interesting.   In Section \ref{secexamples},  we discuss examples for problems in incidence geometry,  exploring the role of number theory and discussing old and then new examples.   There are a number of nice expositions of the new examples, including \cite{ABG+}, \cite{Bl},  \cite{Po}, and \cite{Ta}.  I will give a somewhat different description, explaining first why I did not expect such examples and then how the examples make sense to me in hindsight. 
The complexity of these examples gives one perspective on why the problems are difficult.  In Section \ref{secupperbounds},  we discuss approaches to proving bounds in incidence geometry,  and explore why current bounds are hard to improve.   Besides looking for examples and proving upper bounds, mathematicians in the field also try to classify examples and to understand what structural features examples must possess.  I think this is an important and interesting direction in the field and we discuss it in Section \ref{secgoals}.  

\vskip10pt

{\bf Acknowledgements.}  Thanks to Bjorn Poonen, James Maynard, Melanie Matchett-Wood, Peter Sarnak, and Martin Widmer for helpful comments on drafts of this survey.






\section{Connections to many areas} \label{secconnections}

There are a number of interesting connections between incidence geometry and other areas of math.   I have picked out two connections which I find striking.

In the 1990s, Tom Wolff developed some connections between incidence geometry and problems in harmonic analysis.   For example,  he applied incidence geometry ideas to prove a deep estimate about a problem in PDE called the local smoothing problem.   Consider a solution to the wave equation $u(x,t)$ with initial data $u(x,0) = f(x)$ and $\partial_t u(x,0) = g(x)$.   The problem that Wolff considered was: given bounds on $L^p$-based Sobolev norms of $f$ and $g$ on $\RR^d$,  what can we conclude about $L^p$-based Sobolev norms of the solution $u$ on $\RR^d \times [0,1]$?  The case $p=2$ is classical,  but larger values of $p$ are much more difficult.   Wolff used incidence geometry to prove sharp estimates for certain $p > 2$, giving the first such sharp estimates.

Let us try to give a hint of how incidence geometry comes into play in this problem of PDE and harmonic analysis.   One basic solution to the wave equation models the pressure waves that are generated when you clap your hands.   There is a sphere of high pressure that moves outward from your hands at the speed of sound.  If you imagine this high pressure region in space-time,  it forms a cone with your hands at the center.   Modeling a single clap can be done by a classical explicit solution to the wave equation.   The situations becomes more complex when there are many claps.   Each clap generates a high-pressure cone in spacetime.   The features of the final solution crucially depend on how those cones intersect each other.   Wolff bounded those intersections using a method from incidence geometry called the cutting method,  developed by Clarkson, Edelsbrunner, Guibas, Sharir, and Welzl,  \cite{CEGSW}.  (We will review this work in Section \ref{subsecpart}.)   Combining the incidence geometry with other ideas from Fourier analysis he was able to prove the first sharp estimates about the local smoothing problem.

The problems that appear in harmonic analysis involve thin neighborhoods of cones or lines or circles instead of actual cones or lines or circles.   So the harmonic analysis problems are cousins of the problems in combinatorial geometry.   The harmonic analysis versions of these problems are related to the original combinatorial problems but they involve new issues and serious new difficulties.   Kaufman,  Falconer, and Furstenberg introduced these types of problems in the 1960s and 70s in the context of geometric measure theory and ergodic theory, cf. \cite{Ka},  \cite{Fa},  \cite{Wol}.   There was recently an important milestone in the work on thin-neighborhoods version of incidence geometry: in \cite{OS} and \cite{RW},  Orponen-Shmerkin and Ren-Wang proved the Furstenberg set conjecture,  the natural analogue of Theorem \ref{thmst} with thin rectangles in place of lines.

Recently,  these ideas have been applied in homogenenous dynamics by Lindenstrauss,  Mohammadi, Wang, and Yang starting in \cite{LM}.  They proved strong quantitative estimates for some cases of Ratner's fundamental theorem on the behavior of unipotent orbits in locally symmetric spaces.   

Let us try to give a hint of how incidence geometry comes into play in this problem in dynamics.   Suppose we have a compact manifold $X$ and a diffeomorphism $\phi: X \rightarrow X$.   We have a curve $C \subset X$,  and we want to understand $\phi^n(C) \subset X$ for a large power $n$.    Roughly speaking $\phi^n(C)$ will be a very long curve that wraps around $X$ many times,   and the main question is how $\phi^n(C)$ will be distributed in $X$.    If we intersect $\phi^n (C)$ with a ball $B \subset X$,  we will see many roughly parallel arcs,  and we would like to understand how these arcs are spaced.   Are they roughly evenly spaced or are they clumped together?  We can study this problem inductively -- given some information about $\phi^n(C)$,  we try to describe $\phi^{n+1}(C) = \phi ( \phi^n(C))$.   If $\phi^n(C)$ is clumped together, can we prove that $\phi^{n+1}(C)$ is less clumped together than $\phi^n(C)$?  

In the work of \cite{LM},  $\phi$ is a hyperbolic diffeomorphism,  which means there are some submanifolds,  called stable submanifolds,  which are contracted by $\phi$,  and other submanifolds,  called unstable submanifolds,  which are expanded by $\phi$.   If two points $x_1, x_2 \in \phi^n(C)$ are close together and if they lie on the same stable manifold,  then $\phi(x_1), \phi(x_2)$ will be even closer together than $x_1, x_2$.   If $x_1, x_2 \in \phi^n(C)$ are close together but don't lie on or near the same stable submanifold,  then $\phi(x_1), \phi(x_2)$ will be farther apart than $x_1, x_2$.    So the action of $\phi$ on $\phi^n(C)$ will spread out clumps unless the stable submanifolds intersect $\phi^n(C)$ a lot.  To control $\phi^{n+1}(C)$,  we have to estimate ``incidences'' between stable submanifolds and $\phi^n(C)$.   Roughly speaking,  \cite{LM} estimated this using tools from incidence geometry and used those estimates along with other dynamics tools to give strong quantitative bounds on how $\phi^n(C)$ is distributed in $X$.

\section{Examples} \label{secexamples}

Choosing $n$ points in $\RR^2$ gives us $2n$ real variables to play with.   Enforcing $U$ unit distances leads to $U$ equations in those variables.   If $U > 2n$,  such a system of equations is called overdetermined.   Intuitively,  if $U$ is much bigger than $2n$,  there should only be solutions because of some special structure present in the problem.  

The known examples make use of two types of structure: symmetry and number theory.

\subsection{Symmetry}

The unit distance problem is translation invariant.  If $p_1, p_2,  v \in \RR^2$,  then $|p_1 - p_2| = 1$ if and only if $| (p_1 + v) - (p_2 + v)| = 1$.   We can take advantage of this translation symmetry to build examples with many unit distances.

Suppose that $v_1, ..., v_s$ are unit vectors in $\RR^2$.   Suppose that $a = (a_1, ..., a_s) \in \{0, 1\}^s$ and define $p_a = \sum_{i=1}^s a_i v_i$.   Set $P = \{ p_a \}_{a \in \{0, 1 \}^s}$.   For generic choice of $v_1, ..., v_s$,  $|P| = 2^s$.   Let $n = 2^s = |P|$.

Say that $a,  a' \in \{0, 1 \}^s$ are neighbors if $a_i = a_i'$ for exactly $s-1$ choices of $i$.   If $a,a'$ are neighbors,  then $|p_a - p_{a'}| = 1$.   Each $a \in \{ 0,  1 \}^s$ has $s$ neighborbors,  and so

$$ |U(P)| = s 2^s = (\log_2 n) n. $$

\subsection{Diophantine equations} \label{subsecdioph}

For the unit distance problem and for many other problems in incidence geometry,  the best known examples are based on integer solutions to diophantine equations.   To see why diophantine equations might be natural,  let us back up and look at the incidence geometry problem in a slightly different way.   We can reformulate any incidence geometry problem in the following form.

Suppose that $Z \subset \RR^a \times \RR^b$ is a submanifold or subvariety.   Given finite sets $A \subset \RR^a$ and $B \subset \RR^b$,  define $I_Z(A,B) = Z \cap (A \times B)$.    Define

$$ I_{Z, max}(m,n) = \max_{|A| = m, |B| = n } |I_Z(A,B)|. $$

For example,  in the unit distance problem,  we can let $p_1, p_2, q_1,q_2$ be coordinates on $\RR^2 \times \RR^2$,  and we let $Z$ be the variety defined by

$$ (p_1 - q_1)^2 + (p_2 - q_2)^2 = 1. $$

\noindent Then $I_Z(A,B)$ is the number of unit distances between points $p \in A$ and $q \in B$.   More generally,  for each class of curves/surfaces $C$ we mentioned above,  there is a corresponding variety $Z_C$ so that $I_{C, max} = I_{Z_C, max}$.   

With this minor reformulation,  we can see how diophantine equations could be relevant in incidence geometry.   Suppose that $Z$ is an algebraic variety defined by equations with rational coefficients (which it usually is in examples).    We can choose $A$ and $B$ to be integer vectors with $\ell^\infty$ norm at most $H$,  and then $I_Z(A,B)$ is the set of integer points on $Z$ with $\ell^\infty$ norm at most $H$.   Minor variations of this construction give the best known examples for many problems in incidence geometry.

For future reference,  let us make a little notation.   We write $\ZZ_{\le H} := \{ a \in \ZZ,  |a| \le H \}$.   If $Z \subset \RR^k$ is an algebraic variety,  we write

$$ N_Z(\ZZ, H) = \# \{ \alpha \in \ZZ_{\le H}^k: \alpha \in Z \}. $$

\noindent If $Z \subset \RR^a \times \RR^b$, we can set $A = \ZZ_{\le H}^a$ and $B = \ZZ_{\le H}^b$ and then $I_Z(A,B) = N_Z(\ZZ, H)$.  

Erdos gave an example for the unit distance problem in the 1940s which was the best known example until recently.  The example is a small variation on this method.   We note that by scale invariance,  the unit in unit distance can be replaced by any other number.   So we can replace our original $Z$ by the variety $Z_m$ defined by

$$ (p_1 - q_1)^2 + (p_2 - q_2)^2 = m.  $$

\noindent The variety $Z_m$ is closely related to the variety $S_m \subset \RR^2$ defined by $x_1^2 + x_2^2 = m$.   Then $(p_1, p_2, q_1,q_2) \in Z_m$ if and only if $(p_1 - q_1, p_2 - q_2) \in S_m$.   If $H \ge 10 \sqrt m$,  then a typical point $p \in \ZZ_{\le H}^2$ participates in $N_{S_m}(\ZZ,  H)$ unit distances with other points $q \in \ZZ_{\le H}^2$.    So if we set $A = B = \ZZ_{\le H}^2$,  then we have

$$ | I_{Z_m}(A,B) | = N_{Z_m}(\ZZ,  H) \sim |A| N_{S_m}(\ZZ, H). $$

Next we choose $m$ strategically in order to maximize $N_{S_m}(\ZZ,  H)$.   We need to choose $m \lesssim H^2$ or else $N_{S_m}(\ZZ, H) = 0$.   We will choose $m$ with $m^2 < H$,  so that $N_{S_m}(\ZZ, H) = N_{S_m}(\ZZ)$,  the number of integer solutions to $x^2 + y^2 = m$.   This was studied classically in number theory.   The answer depends on the prime factorization of $m$.   It can be much larger than $\log m$,  but it is bounded by $m^{o(1)}$.   This method gives examples $P \subset \RR^2$ with $|P|=n$ where $|U(P)|$ is significantly bigger than $n \log n$ but still smaller than $C_\eps n^{1 + \epsilon}$ for every $\epsilon > 0$.

Analyzing number theoretic examples of this kind boils down to estimating $N_Z(\ZZ, H)$ for different varieties $Z$.  In the unit distance problem, it boils down to estimating $N_{S_m}(\ZZ, H)$.  There is a closely related problem of estimating the number of integer points on the hyperbola $H_m$ defined by $xy = m$.  

The whole rest of this examples section will be concerned with estimating quantities like $N_Z(\ZZ, H)$.   To start getting perspective, we recall a simple heuristic for guessing the order of magnitude of $N_Z(\ZZ, H)$.  To illustrate the method, suppose that $Z$ is the algebraic variety defined by the equation

$$ x_1^2 + x_2^2 + x_3^2 - x_4^2 - x_5^2 - x_6^2 = 0. $$

\noindent We are trying to count the number of integer solutions to this equation with $|x_i| \le H$.  This bound tells us that $|x_i|^2 \le H^2$ and so the left-hand side lies in $[-3 H^2, 3H^2]$.  Just as a heuristic, let us imagine that the left-hand side is evenly distributed in $[-3H^2, 3H^2]$.  In this case, the fraction of choices $(x_1, .., x_6) \in [-H, H]^6$ that solve the equation would be $\sim H^{-2}$ and so the number of solutions would be $\sim H^{-2} H^6 = H^4$.  For this equation, this simple heuristic is correct, and it can be proven using the circle method (which gives more precise estimates for the number of solutions of this equation).

If $Z$ is the zero set of a single polynomial $P$ with integer coefficients, then this simple heuristic only depends on the number of variables, $n$, and the degree of $P$.  The heuristic is

$$ N_Z(\ZZ, H) \heur H^{n - \deg(P)}. $$

If $Z$ is a hyperbola $H_m$ or circle $S_m$, then we have $n=2$ and $\deg(P) = 2$, and so the simple heuristic above gives

$$ N_Z(\ZZ, H) \heur 1. $$

The actual behavior of $N_{S_m}(\ZZ, H)$ and $N_{H_m}(\ZZ, H)$ is well understood.  Notice that $N_{H_m}(\ZZ)$ is just the number of factors of the integer $m$.  It depends on how $m$ factors into primes.  If $m$ is prime, then $N_{H_m}(\ZZ) = 4$.  If $m$ is a product of $r$ distinct primes, then $N_{H_m}(\ZZ) = 2^{r+1}$.  So as $m$ varies, $N_{H_m}(\ZZ)$ and also $N_{H_m}(\ZZ, H)$ oscillates in a jagged manner.  Nevertheless, the classical divisor bound in number theory states that $1 \le N_{H_m}(\ZZ, H) \le H^{o(1)}$, and so the simple heuristic above gives an accurate answer up to factors $H^{o(1)}$.  The situation for the circle is similar.  There is an exact formula for $N_{S_m}(\ZZ)$ in terms of the prime factorization of $m$.  Using this formula and the divisor bound, we get $0 \le N_{S_m}(\ZZ, H) \le H^{o(1)}$.  And so again the simple heuristic above gives a decent approximation.  

The classical bound for the number of integer points on a circle is important in our story,  so we state it as a theorem in the following form:

\begin{theorem} \label{thmclassdiophcircle} (Classical bound for the number of integer points on a circle) For every $\eps > 0$,  there is $C_\eps$ so that for all $m, H \in \NN$,

$$ N_{S_m}(\ZZ, H) \le C_\eps H^\epsilon.  $$

\end{theorem}

This bound is closely related to (and based on) a similar bound for the number of integer points on a hyperbola.

\begin{theorem} \label{thmclassdiophhyp} (Classical bound for the number of integer points on a hyperbola) Let $H_m$ be the hyperbola defined by $x y = m$.  For every $\eps > 0$,  there is $C_\eps$ so that for all $m, H \in \NN$,

$$ N_{H_m}(\ZZ, H) \le C_\eps H^\epsilon.  $$

\end{theorem}

\subsection{More general number fields} \label{subsecgenfield}

The constructions in the last subsection can be generalized to other number fields.   Suppose that $k$ is a number field and $O_k$ is the ring of integers in $k$.    We need some analogue of $\ZZ_{\le H}$.   There are a few slightly different reasonable definitions.   We will work with algebraic integers with bounded house,  and we recall the definition of the house.   Suppose that $\psi_i: k \rightarrow \RR$ or $\CC$ are the embeddings of $k$ into $\RR$ or $\CC$.   If $k$ is a degree $d$ number field, and it has $r_1$ real embeddings and $r_2$ complex embeddings, then we have $d = r_1 + 2 r_2$.  We can put together all the embeddings $\psi_i$ to form one embedding $\psi: k \rightarrow \RR^{r_1} \times \CC^{r_2}$.  We write $V = \RR^{r_1} \times \CC^{r_2}$, which we view as a $d$-dimensional Euclidean space.  With this setup $\psi(O_k)$ is a lattice in $V$.  

Now we can define the house of $\alpha \in O_k$ by:

\begin{equation} \label{defhouse}
\house(\alpha) = \max_i |\psi_i(\alpha) |
\end{equation}

We write $O_{k, \le H}$ for the set of algebraic integers of house at most $H$:

$$O_{k,  \le H} := \{ \alpha \in O_k: \house(\alpha) \le H \}. $$

If $Z$ is a variety defined by polynomial equations in $n$ variables, then

$$N_Z(O_k, H) := \# \{ (\alpha_1, ..., \alpha_n) \in O_{k, \le H}^n \cap Z \}. $$

We can now make examples as above with $A = O_{k,\le H}^a$ and $B = O_{k, \le H}^b$ and we see that $|I_Z(A,B)| = N_Z(O_k, H)$.   

To the best of my knowledge, these constructions appeared in incidence geometry surprisingly recently.   In the context of point-line incidences,  they appeared in \cite{GS} and \cite{Cu} starting in 2021.   The analysis in \cite{GS} and \cite{Cu} shows that examples for point-line incidences built on number fields work about as well as the older examples built on integers, but no better.

To see why this would happen, let us return to our simple heuristic about $N_Z(\ZZ, H)$ and generalize it to a number field $k$.  First we need to estimate $|O_{k, \le H}|$, and then we can use this to estimate $N_Z(O_k, H)$.

Recall that $\psi(O_k)$ is a lattice in $V$.  Let us call this lattice $\Lambda$.  Note that $O_{k, \le H}$ is just the intersection of $\Lambda$ with a cube of side length $H$.   A little more precisely, if we write $B_{\RR, H}$ for the ball of radius $H$ in $\RR$ and $B_{\CC, H}$ for the ball of radius $H$ in $\CC$, then we can define $Q_H \subset V$ as the ``cube''  $B_{\RR, H}^{r_1} \times B_{\CC, H}^{r_2}$, and then  $O_{k, \le H} = \Lambda \cap Q_H$.  A key character in the story is the covolume of $\Lambda$.  The covolume of $\Lambda$ is closely related to the discriminant of $k$, $\Delta_k$: $ \Cov(\Lambda) = 2^{-r_2} \Delta_k^{1/2}. $   So heuristically, 

$$ |O_{k, \le H} | \heur \frac{\Volume(Q_H) }{\Cov(\Lambda)} \approx \frac{H^d}{\Delta_k^{1/2}}. $$

Now let us return to the example when $Z$ is defined by the equation $x_1^2 + x_2^2 + x_3^2 - x_4^2 - x_5^2 - x_6^2 = 0$.  We want to count the solutions to this equation with $x_i \in O_{k, \le H}$.  The number of choices for the $x_i$ is $|O_{k, \le H}|^6$.  The left-hand side clearly lies in $O_{k, \le 6H^2}$.  So naively, we might expect the fraction of choices $(x_1, ..., x_6)$ that solve the equation to be $\heur |O_{k, \le 6 H^2}|^{-1}$.  Plugging in our heuristic for $|O_{k, \le H}|$, we get

$$ N_Z(O_k, H) \heur |O_{k, \le 6 H^2} |^{-1} |O_{k, \le H}|^6 \approx  \left( \frac{ H^{2d} }{\Delta_k^{1/2} }  \right)^{-1} \left( \frac{H^d}{\Delta_k^{1/2}} \right)^6  = \frac{H^{4d}}{\Delta_k^{5/2}} \heur \frac{1}{\Delta_k^{1/2}} |O_{k, \le H}|^4. $$

As above, if $Z$ is defined by a polonomial in $n$ variables, then this heuristic only depends on the number of variables $n$ and the degree of $P$, and it gives

$$ N_Z(O_k, H) \heur \frac{1}{\Delta_k^{ \frac{ \deg (P) - 1}{2}}} |O_{k, \le H} |^{n - \deg P}. $$

We recall that $\Delta_k$ is a positive integer, and so $\Delta_k \ge 1$, and so this heuristic gives $N_Z(O_k, H) \lessapprox |O_{k, \le H} |^{n - \deg P}$.  This heuristic predicts that number fields of all degrees give basically similar examples, but there is a constant factor out front that gets smalelr as the discriminant goes up.

In particular, if $Z$ is a circle or hyperbola, then $n=2$ and $\deg P = 2$ and our heuristic reads

$$ N_{S_m}(O_k, H), N_{H_m}(O_k, H) \heur \frac{1}{ \Delta_k^{1/2}}. $$

Number fields also appeared in computational work on the unit distance problem.    In \cite{Eng},  Engel et al showed that finite subsets of the `Moser lattice' optimize $U_{max}(n)$ for $n \le 15$.   They point out that the Moser lattice is the ring of integers of a degree 4 number field.   A little later in \cite{Ale},  Alexeev et al computed $U_{max}(n)$ for $n \le 21$ and observed that the optimal configurations are still finite subsets of the Moser lattice.   It is quite challenging to compute $U_{max}(n)$ exactly,  as we will discuss more later,  and $U_{max}(21)$ is the state of the art.   This limited data might suggest that number fields would play a role in the optimal unit distance examples, but it was hard to guess how the example would generalize as $n \rightarrow \infty$.  

\subsection{Recent examples} \label{subsecrecent}

Recently Open AI gave a counterexample to the unit distance conjecture by showing that $N_{S_m}(O_k, H)$ can be far larger than predicted by the simple heuristic above.  

\begin{theorem} \label{countercircle} (Open AI, \cite{OAI}) There is an exponent $\beta > 0$ so that the following holds.   We can choose a number field $k$, 
a non-zero $m \in k$,  and a positive number $H$ so that 

\begin{itemize}

\item $|O_{k, \le H}|$ is arbitarily large.

\item $N_{S_m}(O_k,H) \ge |O_{k, \le H}|^\beta$.

\end{itemize}

\end{theorem}

This theorem shows that the classical upper bound for the number of integer points on a circle, Theorem \ref{thmclassdiophcircle}, does not generalize to arbitrary number fields in a uniform way.   Theorem \ref{countercircle} is the main ingredient in the counterexample to the unit distance conjecture.

A little later, inspired by this example,  Bloom,  Sawin,  Schildkraut, and Zhelezov found a similar counterexample for hyperbolas.

\begin{theorem} \label{counterhyp} (\cite{BSSZ}) There is an exponent $\beta > 0$ so that the following holds.   We can choose a number field $k$, 
a non-zero $m \in k$,  and a positive number $H$ so that 

\begin{itemize}

\item $|O_{k, \le H}|$ is arbitarily large.

\item $N_{H_m}(O_k,H) \ge |O_{k, \le H}|^\beta$.

\end{itemize}

\end{theorem}

In fact,  \cite{BSSZ} proved that we can take $m=1$ in the above theorem.   This special case is particularly important and they use it to give a counterexample to the sum-product conjecture in combinatorial number theory.  

The proofs in \cite{OAI} and \cite{BSSZ} are based on two different mechanisms.   The proof in \cite{OAI} is based on constructing a number field where $m$ factors in a particular way.  The proof in \cite{BSSZ} is based on studying the units of $k$.  
 A paper by Mythos \cite{Anth} gave an alternate proof of Theorem \ref{countercircle} using units in a manner similar to \cite{BSSZ}.   See also Thomas Bloom's blog \cite{Bl} for a nice discussion of these approaches.  
 
 The size of the discriminant $\Delta_k$ will play a key role in the proofs of these theorems.  Before we turn to the proofs, we discuss this.
 
\subsection{How big is the discriminant of a degree $d$ number field} \label{subsecdiscr}

According to our heuristic above, $N_Z(O_k, H)$ tends to be larger when $\Delta_k$ is small.  Indeed it turns out to be important to choose $\Delta_k$ as small as possible.  This leads to the question: how small  the discrimant can be for a degree $d$ field.    Let us write $\Delta_{min}(d)$ for the infimal discriminant of a degree $d$ number field.  As we mentioned above, $\Delta_k$ is a positive integer, and so $\Delta_{min}(d) \ge 1$.  Minkowski proved that there is a constant $C_{Mink} > 1$ so that $\Delta_{min}(d) > C_{Mink}^d$.   On the other hand,  a degree $d$ cyclotomic field has discriminant around $d^d$ and so $\Delta_{min}(d)$ grows at most like $d^d$.    It was an open question for a long time whether $\Delta_{min}$ grows at most exponentially.   I have heard that many people believed that $\Delta_{min}$ should grow superexponentially in $d$ cf. page 124 in \cite{O}.   

The question was resolved by Golod-Shafarevich in the late 1960s.   Golod and Shafarevich were studying towers of unramified field extensions.  We do not define unramified here, but it is a basic property of field extensions.  If $k_2$ is a degree $e$ field extension of a number field $k_1$, then $\Delta_{k_2} \ge \Delta_{k_1}^e$.  If $k_2$ is an unramified extension, then $\Delta_{k_2} = \Delta_{k_1}^e$, and if $k_2$ is a ramified extension of $k_1$ then $\Delta_{k_2} > \Delta_{k_1}^e$.  Golod and Shafarevich proved that there is a number field $k_1$ and an infinite tower of unramified finite extensions $k_1 \subset k_2 \subset ...$.  As a corollary, they proved the following result about discriminants of high degree number fields:

\begin{theorem} \label{thmmart} (Golod-Shafarevich,  \cite{GS},  1960s) There is a constant $C_{GS}$ and a sequence of number fields $k$ with degree $d \rightarrow \infty$ so that $\Delta_k < C_{GS}^d$.  

We write $r_1(k)$ for the number of real embeddings of $k$ and $r_2(k)$ for the number of complex embeddings of $k$.  For any rational number $q$, there is a constant $C_{GS,q}$ and a sequence of number fields $k$ with degree $d \rightarrow \infty$ so that $r_2(k) = q r_1(k)$ and $\Delta_k < C_{GS,q}^d$.   \end{theorem}

Remark: This theorem is an immediate corollary of the work of Golod and Shafarevich.  It is stated explicitly by Martinet in \cite{Ma}, which also gave explicit bounds for the constant.
  These number fields are deep and remarkable examples.  There are several variations on the proofs of Theorems \ref{countercircle} and \ref{counterhyp}.  All variations crucially use Theorem \ref{thmmart}.

\subsection{Estimating the number of $O_k$ points on varieties more carefully} \label{subsecdir}

In this subsection,  we sketch one approach to proving Theorems \ref{countercircle} and \ref{counterhyp}.  Our proof is inspired by the proofs of \cite{BSSZ} and \cite{Anth} using units, but we write it in a somewhat different way.   One of our goals is to give some intuition why these curves contain so many $O_k$ points with small house.  We will try to explain what was wrong with the simple heuristic from the last sections.

As a warmup,  let us consider the number of integer points on a hyperbola $H_m$.    The number of integer points on the hyperbola $H_m$ is the number of divisors of $m$,  and it has been extensively studied in number theory.   One classical result is Dirichlet's theorem on the average size of $N_{H_m}(\ZZ)$.   

\begin{theorem} (Dirichlet) 

$$ \frac{1}{M} \sum_{m=1}^M N_{H_m}(\ZZ) \sim \log M. $$

\end{theorem}

Dirichlet actually proved much more refined estimates,  but this will suffice for our purpose.    Here is some intuition and a proof sketch for this theorem.  Note that $\sum_{m=1}^M N_{H_m}(\ZZ)$ is equal to the number of integer lattice points in the region $R_M$ defined by

$$ R_M := \{ (x,y) \in \RR^2: 1 \le x,y \le M,  1 \le xy \le M \}. $$

\noindent The area of $R_M$ is about $M \log M$.   Therefore,  we might expect that $R_M$ contains $\sim M \log M$ integer points.  To find the area of $R_M$, we could compute by an integral.  But to get intuition for future arguments, we also note that we can choose disjoint rectangles in $R_M$ with total area $\sim M \log M$.  Specifically, if $2 \le a \le M/2$ is dyadic, we can let $\Rect(a)$ be the rectangle $\{ (x,y) \in \RR^2: a/2 \le x \le a, \frac{M}{2a} \le y \le \frac{M}{a} \}$.  Each of these rectangles has area $\sim M$, and the number of rectangles is $\sim \log M$.  Also each rectangle $\Rect(a)$ contains $\sim M$ lattice points $(x,y)$, and so the number of lattice points in $R_M$ is $\gtrsim M \log M$. 

We can view Dirichlet's argument as a more refined version of the simple heuristic from Section \ref{subsecdioph}.   If $x,y \in \ZZ_{\le H}$,  then the heuristic in Section \ref{subsecdioph} said that $xy$ is evenly distributed in $\ZZ_{\le H^2}$.  However,  $xy$ is not evenly distributed in $\ZZ_{\le H^2}$.   One bias is that small numbers in $\ZZ_{\le H^2}$ are over-represented.   For instance, if we choose $x,y$ randomly in $\ZZ_{\le H}$,  the probability that $|xy| \le H$ is around $\frac{ \log H}{H}$, not around $\frac{1}{H}$.  
Dirichlet's argument measures and uses this bias to give a more accurate estimate for $N_{H_m}(\ZZ)$.  

Next we will generalize this average case argument to number fields $O_k$.   It will give a more refined version of the simple heuristic from Subsection \ref{subsecgenfield}.  While the simple heuristic does not predict Theorems \ref{countercircle} and \ref{counterhyp}, this more refined heuristic predicts that these theorems hold when $H$ is a large constant and $\Delta_k < C_{GS}^d$.   In Section \ref{secrigdir},  we will convert these heuristics to rigorous arguments and give alternate proofs of Theorem \ref{countercircle},  Theorem \ref{counterhyp},  and Theorem \ref{counterunit}.

We now apply Dirichlet's method with algebraic integers from $O_k$.   Instead of estimating $\sum_{m=1}^M N_{H_m}(\ZZ)$,  we will estimate

$$ \sum_{m \in O_{k, \le H} } N_{H_m} (O_k,  H) = \# \{ x, y \in O_k: \house(x),  \house(y),  \house(xy) \le H \} $$

\noindent Recall that $\psi_i: k \rightarrow \RR$ or $\CC$ are the embeddings of $k$ into $\RR$ or $\CC$ and that $\house(x) = \max_i |\psi_i(x)|$.   Also $\psi_i(xy) = \psi_i(x) \psi_i(y)$.  Therefore,  we get

$$ \sum_{m \in O_{k, \le H} } N_{H_m} (O_k,  H) = \# \{ x, y \in O_k: |\psi_i(x)|, |\psi_i(y)|,  |\psi_i(x) \psi_i(y)| \le H \textrm{ for all } i   \} $$

Recall that $\Lambda = \psi(O_k)$ is a lattice.   The right hand side of the last equation counts the number of points of $\Lambda^2$ that lie in an appropriate region $R_H$.   To write down $R_H$ in coordinates,  it is a bit simpler to consider the case that $k$ is totally real so that $\psi_i: k \rightarrow \RR$ for all $i = 1, .., d$.   Then the region $R_H$ is  defined by 

\begin{equation} \label{defRH}
R_H := \{ (x_1, ..., x_d, y_1, ... y_d) \in \RR^{2d}: |x_i|,  |y_i|,  |x_i y_i| \le H \textrm{ for each } i \},
\end{equation}

To summarize we have

\begin{equation} \label{useRH}
\sum_{m \in O_{k, \le H} } N_{H_m}(O_k, H) = \# (\Lambda^2 \cap R_H).
\end{equation}

This gives the heuristic

\begin{equation} \label{heurgn}
\sum_{m \in O_{k, \le H} } N_{H_m}(O_k, H) \heur \frac{ \Volume(R_H)} {\Cov(\Lambda^2)} \approx \frac{ \Volume(R_H)} {\Delta_k}. 
\end{equation}

The region $R_H$ is a Cartesian product: 

$$ R_H = \prod_{i=1}^d \{ (x_i, y_i) \in \RR_2: |x_i|, |y_i|, |x_i y_i| \le H \}. $$

\noindent Each factor has area $\sim H \log H$,  just as in Dirichlet's original theorem.    Therefore,  the volume of the region $R_H$ is $\approx (H \log H)^d$,  and so heuristically 

\begin{equation} \label{heurgn}
\sum_{m \in O_{k, \le H} } N_{H_m}(O_k, H) \heur \frac{H^d (\log H)^d}  {\Delta_k}. 
\end{equation}

(We remark that the sum on the left includes $m=0$,  which we would like to leave out.   The contribution from $m=0$ is $2 |O_{k, \le H}| \approx H^d / \Delta_k^{1/2}$.  So as long as $(\log H)^d$ dominates $\Delta_k$,  this contribution is negligible.)

Next we would like to replace the sum by an average.   Recall our heuristic that $|O_{k, \le H}| \heur \frac{ H^d}{\Delta_k^{1/2}}$.    Plugging in, we get the following heuristic for the average value of $N_{H_m}(O_k, H)$:

$$ \Average_{m \in O_{k, \le H} } N_{H_m}(O_k,H) \heur \frac{ (\log H)^d }{\Delta_k^{1/2}}. $$

If we fix $k$ and let $H \rightarrow \infty$,  then this is tiny compared to $|O_{k, \le H}| \sim H^d$.   However,  there is a more crafty choice: if we fix $H$ and consider fields $k$ with degree $d \rightarrow \infty$.   When we do so,  we have to be careful about the discriminant.    Here it is crucial to use number fields with $\Delta_k$ growing at most exponentially in $d$,  and these are provided by Golod-Shafarevich-Martinet (Theorem \ref{thmmart}).  If we use number fields $k$ with $\Delta_k^{1/2} < C_{GS}^d$,  then our heuristic becomes

$$ \Average_{m \in O_{k, \le H} } N_{H_m}(O_k,H) \sim \frac{ (\log H)^d }{C_{GS}^d} = \left( \frac{ \log H}{C_{GS}} \right)^d. $$

\noindent We can choose a large constant $H$ so that $\log H \gg C_{GS}$,  and then our heuristic gives a small positive exponent $\beta$ so that 

$$ \Average_{m \in O_{k, \le H} } N_{H_m}(O_k,H) \gtrsim |O_{k, \le H}|^\beta. $$

Turning this heuristic into a rigorous argument requires some tools from the geometry of numbers,  but it is not too difficult.   We give a rigorous proof in Section \ref{secrigdir}.  

We can adapt the same method to deal with the circle.   Recall that $\psi_i: k \rightarrow \RR$ or $\CC$ are the embeddings of $k$ into $\RR$ or $\CC$.   This time,  we will consider embeddings of both kinds.   We can group them into an embedding $\psi: k \rightarrow \RR^{r_1} \times \CC^{r_2}$,  where $d = r_1 + 2 r_2$.   Again $\psi(k)$ is a lattice $\Lambda$ in this space with covolume $2^{-r_2} \Delta_k^{1/2}$.   To get a counterexample to the unit distance conjecture,  we will need $r_1 \ge 1$ and $r_2 \sim d$.   By taking $r_1 \ge 1$ and fixing a real embedding $\psi_1: k \rightarrow \RR$,  we can regard $k$ as a subset of $\RR$ so that we ultimately get a set of points $P \subset \RR^2$ in the unit distance conjecture.   But we will see that the complex embeddings are crucial to getting a large number of $O_{k, \le H}$ points on an average circle $S_m$.   

When we imitate the analysis above,  the region $R_H$ becomes

\begin{equation} \label{defRHt} \tilde R_H = \{ (x,y) \in \RR^2: |x|, |y|, |x^2 + y^2| \le H \}^{r_1} \times \{ (z,w) \in \CC^2: |z|, |w|,  |z^2 + w^2| \le H \}^{r_2}. \end{equation}

As above,  we have the heuristic

\begin{equation} \label{heurgn}
\sum_{m \in O_{k, \le H} } N_{H_m}(O_k, H) \heur \frac{ \Volume(\tilde R_H)} {\Cov(\Lambda^2)} \approx \frac{ \Volume(\tilde R_H)} {\Delta_k}. 
\end{equation}

We need to estimate $\Volume(\tilde R_H)$.  Since $\tilde R_H$ is a product,  we should estimate the volume of each factor.  
The factor $ \{ (x,y) \in \RR^2: |x|, |y|, |x^2 + y^2| \le H \}$ is just a disk with area around $H$.   We turn to the other factor,  which we call $U =  \{ (z,w) \in \CC^2: |z|, |w|,  |z^2 + w^2| \le H \} $.   The key point is that $U$ has volume around $H^2 \log H$.   To see this,  we can factor $z^2 + w^2 = (z + iw) (z - i w)$.   If we define $u = z+iw$ and $v = z - iw$,  then the set $U$ is essentially the same as $\{ (u,v) \in \CC^2: |u|, |v|, |uv| \le H \}$.   If $1 \le \lambda \le H$ is a dyadic number,  let $U_{\lambda} = \{ (u,v) \in \CC^2: |u| \sim \lambda, |v| \sim H/ \lambda \}$.   Then the sets $U_\lambda$ are disjoint subsets of $U$.   Each set $U_\lambda$ has volume around $H^2$,  and the number of choices for $\lambda$ is around $\log H$.   Therefore,  the volume of $U$ is around $H^2 \log H$,  and so the volume of $\tilde R_H$ is approximately:

$$ \Volume(\tilde R_H) \sim H^d (\log H)^{r_2}. $$

This gives the heuristic

$$ \Average_{m \in O_{k, \le H} } N_{S_m}(O_k,H) \sim \frac{ (\log H)^{r_2} }{\Delta_k^{1/2}}. $$

If $r_2 \sim d$,  then this bound works as well as the bound we had in the hyperbola case.   This finishes our heuristics for $N_{S_m}(O_k, H)$.   Turning these heuristics into a rigorous argument requires some tools from the geometry of numbers,  which we discuss in Section \ref{secrigdir}.  

\subsection{Other examples}

The examples we described in Section \ref{subsecdir} are closely related to the examples in \cite{BSSZ} and \cite{Anth}, based on units.  The reason is that the Dirichlet unit theorem has a proof which is related to the arguments above, and unwinding that proof more or less leads to the method in Section \ref{subsecdir}.

The other examples are based on building a number field $k$ where a particular number $m \in \NN$ factors in a favorable way.  This argument has a different mechanism.  In this setting, $N_{S_m}(O_k, H)$ is much larger than $\Average_{m \in O_{k, \le H} } N_{S_m}(O_k, H)$.  This method is well explained in \cite{Ta} and \cite{ABG+}.  The method is also interesting and it leads to stronger bounds for the exponents in the theorems.

Let us also mention again that \cite{BSSZ} and \cite{Po1} use high degree number fields in related but different ways to give other remarkable examples in combinatorial geometry and number theory.  We don't describe these examples in detail in order to keep this survey from getting too long, and also because they are well described in the introductions of those papers.

\section{Rigorous arguments in the geometry of numbers} \label{secrigdir}

In this section,  we give a rigorous version of the heuristic argument from Subsection \ref{subsecdir}.   We use it to prove Theorems \ref{counterhyp} and \ref{countercircle}, saying that there can be surprisingly many $O_k$ of bounded house on a hyperbola or circle,  and also to prove Theorem \ref{counterunit} giving a counterexample to the unit distance conjecture.  The later sections are independent of this discussion, so you can also skip this section and read about upper bounds and classification problems.  

We recall some notation.  Let $k$ be a number field with ring of integers $O_k$.   For $m \in O_k$,  we let $S_m$ be the variety defined by $x^2 + y^2 = m$ and $H_m$ the variety defined by $x y = m$.   We are interested in the number of $O_k$-points on these varieties with house at most $H$.   (Recall that if $\psi_i: k \rightarrow \RR$ or $\CC$,  then $\house(\alpha) = \max_i |\psi_i(\alpha)|$.)  We write $O_{k, \le H}$ for the integers in $O_k$ with house at most $H$.   Then we define

$$ N_{S_m}(O_k, H) = \# ( O_{k, \le H}^2 \cap S_m),  N_{H_m}(O_k, H) = \# (O_{k, \le H}^2 \cap H_m ). $$

We will study the average size of these numbers as $m$ varies in $O_{k, \le H}$.   This question is analogous to Dirichlet's theorem about the average size of the divisor function.   

Suppose that $k$ has degree $d$ and has $r_1$ real embeddings and $r_2$ complex embeddings with $r_1 + 2 r_2 = d$.  We can conbine these to give a map $\psi: k \rightarrow \RR^{r_1} \times \CC^{r_2}$, and we let $\Lambda = \psi(O_k) \subset V  = \RR^{r_1} \times \CC^{r_2}$.  Then we can interpret $\sum_{m \in O_{k, \le H}} N_{H_m}(O_k, H)$ as the number of points of $\Lambda^2$ in a certain region $R_H$ of $V^2$.   As a heuristic,  we might guess that the number of lattice points in this region is $\heur \frac{\Volume(R_H)}{\Cov(\Lambda^2)}$.   After estimating $\Volume(R_H)$,  this leads to the following heuristics on the average behavior of $N_{H_m}$ and $N_{S_m}$:

\begin{equation} \label{heurhyp} \Average_{ m \in O_{k, \le H} } N_{H_m}(O_k, H) \heur \frac{ (\log H)^{r_1 + r_2} }{\Delta_k^{1/2}}. \end{equation}

\begin{equation} \label{heurcir} \Average_{m \in O_{k, \le H} } N_{S_m}(O_k, H) \heur \frac{ (\log H)^{r_2} }{\Delta_k^{1/2}}. \end{equation}

In this section we state and prove rigorous theorems giving lower bounds that are fairly close to these heuristics.   It is simplest to make the heuristics above rigorous in the regime where we fix $k$ and let $H \rightarrow + \infty$.  

\begin{theorem} \label{thmlargeHasymp}There are constants $c_1, c_2, \tilde c_1, \tilde c_2 > 0$ so that the following holds.  Suppose that $k$ is a number field.  Then, as $H \rightarrow + \infty$,  

\begin{equation} \label{largeHhyp} \Average_{ m \in O_{k, \le H} } N_{H_m}(O_k, H) = (1 + o(1)) c_1^{r_1} c_2^{r_2} \frac{ (\log H)^{r_1 + r_2} }{\Delta_k^{1/2}}. \end{equation}

\begin{equation} \label{largeHcir} \Average_{m \in O_{k, \le H} } N_{S_m}(O_k, H) = (1 + o(1)) \tilde c_1^{r_1} \tilde c_2^{r_2}  \frac{ (\log H)^{r_2} }{\Delta_k^{1/2}}. \end{equation}

\end{theorem}

The proof of this theorem is a straightforward application of the geometry of numbers.  Since we will not use the theorem, we omit the proof.  (I have tried to determine whether this result or something similar has appeared in the mathematical literature.  I'm not positive.  I have not been able to find such a result.  But on the other hand, the ideas involved are quite classical.)
  
   For the application to the unit distance problem,  we need to work in a different regime: where $H$ is a large constant and where the bounds depend in a controlled way on the choice of field $k$.   We work out a rigorous result that applies in this regime.  The bounds are not as sharp, but the hypotheses are more flexible.  
\begin{theorem} \label{thmlowerhyp} There is a constant $C > 0$ so that the following holds.  Suppose that $k$ is a degree $d$ number field with $r_1$ real embeddings and $r_2$ complex embeddings, then

\begin{equation} \label{goal1}
 \sum_{m \in O_{k,H}} N_{H_m}(O_k, H) \ge C^{-d} \frac{ H^d (\log H)^{r_1 + r_2} }{\Delta_k^{1.01}}. 
 \end{equation}

If also $\log H > C \Delta_k^{\frac{10}{d}}$, then

\begin{equation}  \label{goal2}  \Average_{0 \not= m \in O_{k, \le H} } N_{H_m}(O_k, H) \ge C^{-d} \frac{ (\log H)^{r_1 + r_2} }{ \Delta_k^{1.01}}.  \end{equation} 

\end{theorem}

There is a similar theorem with $S_m$ in place of $H_m$:

\begin{theorem} \label{thmlowercir} There is a constant $C > 0$ so that the following holds.  Suppose that $k$ is a degree $d$ number field with $r_1$ real embeddings and $r_2$ complex embeddings and with $r_2 \ge d/10$, then

\begin{equation} \label{goal1}
 \sum_{m \in O_{k,H}} N_{S_m}(O_k, H) \ge C^{-d} \frac{ H^d (\log H)^{r_2} }{\Delta_k^{1.01}}. 
 \end{equation}

If also $\log H > C \Delta_k^{\frac{20}{d}}$, then

\begin{equation}  \label{goal2}  \Average_{0 \not= m \in O_{k, \le H} } N_{S_m}(O_k, H) \ge C^{-d} \frac{ (\log H)^{r_2} }{ \Delta_k^{1.01}}.  \end{equation} 

\end{theorem}

Comparing these theorems to the heuristics above,  we see that the power of $\Delta_k$ in the theorems is worse than in the heuristic.   In the regime where $\Delta_k < C^d$,  this difference is not that significant,  and the lower bounds in the theorems match the heuristics up to factors of the form 
$C^{-d}$.  

Theorem \ref{thmlowerhyp} and \ref{thmlowercir} require a little more care than Theorem \ref{thmlargeHasymp}, but they still follow from classical methods in the geometry of numbers.  

We recall some standard tools from the geometry of numbers that we will use to prove these theorems, beginning with a lower bound on the number of lattice points in a symmetric convex set.  Recall that a convex set $K$ is called symmetric if, for every $x \in K$, $-x \in K$.  

\begin{lemma} \label{lemlower} (Blichfeldt and van der Corput) If $V$ is a $d$-dimensional Euclidean space and $K \subset V$ is a  symmetric convex set and $\Lambda \subset V$ is a lattice  then

$$ \# ( \Lambda \cap K) \ge 2^{-d} \frac{\Volume(K)}{\Cov(\Lambda)}. $$

\end{lemma}

\begin{proof} Let $(1/2) K$ be the dilate of $K$ by a factor $(1/2)$.   Consider the standard quotient map $(1/2) K \rightarrow V / L$.   By the pigeonhole principle,  we can choose $y \in V / \Lambda$ so that 

$$ \# \{ x \in (1/2) K: x + \Lambda = y \} \ge \frac{ \Volume( (1/2) K) }{\Cov(\Lambda)} = 2^{-d} \frac{\Volume(K)}{\Cov(\Lambda)}. $$

Let $\{ x_1, ..., x_N \}$ be the set on the left hand side.   Since $K$ is a symmetric convex set,  $x_1 - x_n \in \Lambda \cap K$ for each $n = 1... N$.  To see this, note that $x_1, x_n \in (1/2) K$.  Since $K$ is symmetric, $-x_n$ is in $(1/2) K$.  By convexity, $(1/2)( x_1 - x_n) \in (1/2) K$, and so $x_1 - x_n \in K$.  

\end{proof}

Remark: This lemma may not hold if $K$ is a translation of a symmetric convex set, such as a rectangle whose center is not at zero.

We also note that this lemma applies to both $\Lambda \subset V$ and $\Lambda^2 \subset V^2$.   

Next we will prove upper bounds for the number of points from $\psi(O_k)$ in regions.  To prepare, we begin with a lemma about the shortest non-zero vector in $\psi(O_k)$:

\begin{lemma} \label{lemshortvec} If $k$ is a degree $d$ number field and $0 \not= \alpha \in O_k$,  then $|\psi(\alpha)| \ge (1/2) \sqrt{d}$.
\end{lemma}

\begin{proof}  Recall that the norm of $\alpha$ is the product of all conjugates of $\alpha$.  Suppose that $\psi_i: k \rightarrow \RR$ for $i = 1, ..., r_1$ and $\psi_i: k \rightarrow \CC$ for $i = r_1 + 1, ..., r_2$. 

$$ N(\alpha) = \prod_{i=1}^{r_1} \psi_i (\alpha) \prod_{i=r_1+1}^{r_2}  \psi_i(\alpha)  \prod_{i=r_1+1}^{r_2}  \overline{ \psi_i(\alpha) } . $$

The norm of $\alpha$ is a non-zero integer,  and so 

$$ 1 \le = \prod_{i=1}^{r_1} |\psi_i (\alpha)| \prod_{i=r_1+1}^{r_2} | \psi_i(\alpha) | \prod_{i=r_1+1}^{r_2}  | \psi_i(\alpha)| . $$

This product has $r_1 + 2 r_2 = d$ factors.  By the arithmetic-geometric mean inequality,  we get

$$ d \le \sum_{i=1}^{r_1} |\psi_i(\alpha)|^2 + 2 \sum_{i=r_1+1}^{r_2} |\psi_i(\alpha)|^2. $$

Therefore $|\psi(\alpha)|  \ge (1/2) \sqrt{d}$.
\end{proof}

Next we give an upper bound for the number of lattice points in regions.   Let $\Cover_s(U)$ be the minimal number of cubes of side length $s$ needed to cover $U$.  

\begin{lemma} \label{lemupper} If $V$ is a $d$-dimensional Euclidean space and $\Lambda \subset V$ is a lattice where $|v| \ge (1/10) \sqrt d$ for each non-zero $v \in \Lambda$, and if $U \subset V$,  then $ \# (\Lambda \cap U) \le \Cover_{1/10}(U)$. 

\end{lemma}

\begin{proof} If $Q$ is a cube of side length $1/10$, and $v_1, v_2 \in \Lambda \cap Q$, then $v_1 - v_2$ is a non-zero vector $v \in \Lambda$ with $|v| < (1/10) \sqrt{d}$, contradiction.  Therefore, $\# (\Lambda \cap Q) \le 1$.

Now if $U \subset \cup_i Q_i$ and each $Q_i$ is a cube of side length $1/10$, then $\# (\Lambda \cap U)$ is bounded by the number of cubes.

\end{proof}

In particular, combining Lemmas \ref{lemlower} and \ref{lemupper}, we get basic upper and lower bounds for $|O_{k, \le H}|$.  

\begin{lemma} \label{lemsizeOkH}
If $H \ge 1$, then

\begin{equation} \label{sizeOkH}
2^{-d} \frac{H^d}{\Delta_k^{1/2}} \le |O_{k, \le H} | \le C_1^d H^d. 
\end{equation}

\end{lemma}

Now we begin the proof of Theorem \ref{thmlowerhyp}. 

\begin{proof}
Let $\psi_i: k \rightarrow \RR$ or $\CC$ be the embeddings of $k$.  Here $\psi_i: k \rightarrow \RR$ if $1 \le i \le r_1$ and $\psi_i: k \rightarrow \CC$ if $r_1 + 1 \le i \le r_1 + r_2$.  We abbreviate $r = r_1 + r_2$, so $i$ goes from 1 to $r$.  We let $V = \RR^{r_1} \times \CC^{r_2}$.  We let $x = (x_1, ..., x_r) \in V$, where $x_i \in \RR$ if $1 \le i \le r_1$ and $x_i \in \CC$ if $r_1 + 1 \le i \le r_1 + r_2 = r$.  We let $\Lambda = \psi(O_k) \subset V$ and recall that $\Cov(\Lambda) = 2^{-r_2} \Delta_k^{1/2}$.

We note that $\sum_{m \in O_{k, \le H}} N_{H_m}(O_k, H)$ is the number of $\Lambda^2$ points in the region $R_H \subset V^2$ defined by

\begin{equation} \label{defRHrig}
R_H = \{ (x,y) \in V^2: |x_i|, |y_i|, |x_i y_i| \le H \}. 
\end{equation}

We want to lower bound $\# (\Lambda^2 \cap R_H)$.  We cannot directly apply Lemma \ref{lemlower} because $R_H$ is not convex, so we first locate some convex sets inside $R_H$.  Suppose that $a = (a_1, ..., a_r)$, with $1 \le a_i \le H$.  We define $K(a)$ by

$$ K(a) = \{ (x,y) \in V \times V: |x_i| \le a_i, |y_i| \le \frac{H}{a_i} \textrm{ for all } i \}. $$

Note that $K(a) \subset R_H$.  Since $K(a)$ is a symmetric convex set, Lemma \ref{lemlower} gives

\begin{equation} \label{lowerKa}
\# (\Lambda^2 \cap K(a) ) \ge 2^{-2d} \frac{ \Volume (K(a))}{\Cov(\Lambda^2)} \ge C^{-d} \frac{ H^d}{\Delta_k} .
\end{equation}

We consider $a = (a_1, ..., a_r)$ with $1 \le a_i \le H$ and $a_i$ dyadic for each $i $.  The number of choices for such $a$ is $\sim (\log H)^r$.  

Unfortunately, the sets $K(a)$ are not disjoint.  If these sets $K(a)$ were disjoint, then \eqref{lowerKa} would imply the strong bound $\# (\Lambda^2 \cap R_H) \ge C^{-d} \frac{ H^d (\log H)^r}{\Delta_k}$ which is stronger than \eqref{goal1}.  We could choose rectangles $\Rect(a) \subset K(a)$ so that the rectangles $\Rect(a)$ are disjoint and $\Volume(\Rect(a)) \ge C^{-d} \Volume (K(a))$.  However, the rectangles $\Rect(a)$ would not be centered at the origin, and so  the lower bound Lemma \ref{lemlower} would not apply to them.  Therefore, we consider how the sets $K(a)$ intersect each other.  

To get a handle on these intersections, we decompose $R_H$ into subsets where $|x_i|$ and $|y_i|$ are in dyadic ranges.  For $(a,b) = (a_1, ..., a_r, b_1, ..., b_r)$, we define $\tilde K(a,b)$ by

$$ \tilde K(a,b) = \{ (x,y) \in V^2: |x_i| \sim a_i \textrm{ and }  |y_i| \sim b_i \}. $$

\noindent We consider dyadic $a_i, b_i$ with $1 \le a_i, b_i$ and $a_i b_i \le H$.  (If $a_i = 1$, we replace $|x_i| \sim 1$ by $|x_i| \le 1$.)  Then we have

$$ R_H = \sqcup_{a,b} \tilde K(a,b). $$

Each point of $\tilde K(a,b)$ lies in  $\sim \mu(a,b)$ different sets $K(\tilde a)$, where

\begin{equation} \label{muform}
 \mu(a,b) \sim \prod_{i=1}^r \log \left( \frac{ H }{a_i b_i} \right). 
 \end{equation}
 
 Therefore,
 
 \begin{equation} \label{frame} C^{-d} \frac{ H^d (\log H)^r }{\Delta_k} \le \sum_a \# (\Lambda^2 \cap K(a)) \sim  \sum_{a,b} \mu(a,b) \# (\Lambda^2 \cap \tilde K(a,b) ). \end{equation} 
 
 The following lemma will help to organize the sum on the right-hand side.  
 
 \begin{lemma} Recall that $(a,b) = (a_1, ..,. a_r, b_1,..., b_r)$ with $a_i, b_i$ dyadic in the range $1 \le a_i, b_i \le H$.   If $\tilde \mu \ge 2$,  then the number of such $(a,b)$ with $\mu(a,b) \le \tilde \mu^r$ is bounded by
 
 $$C^r (\log H)^r \tilde \mu^r (\log \tilde \mu)^r . $$
 
  \end{lemma}

\begin{proof} There are at most $(\log H)^r$ choices of $a$.  We fix $a$ and consider the choices for $b$.

We write $b_i = 2^{-k_i} \frac{H}{a_i}$, so that $\mu(a,b) = \prod_{i=1}^r k_i$.  So we need to count the number of choices for $k_i \in \NN$ so that $\prod_{i=1}^r k_i \le \tilde \mu^r$.

Suppose that $2^{\ell_i} \le k_i  < 2^{\ell_i +1}$, where $\ell_i \in \ZZ_{\ge 0}$.   Then $\prod_i 2^{\ell_i} \le \prod_i k_i \le \tilde \mu^r$.   If we fix $\ell_i$, then there are $\le 2^{\ell_i}$ choices for $k_i$, and so the number of choices for $k$ is $\le \prod_i 2^{\ell_i} \le \tilde \mu^r$.

Now we need to consider the number of choices of $\ell = (\ell_1, ..., \ell_r)$.  We know that $\ell_i \in \ZZ_{\ge 0}$ and that $\sum_{i=1}^r \ell_i \le r \log_2 (\tilde \mu)$.   Let $T_s$ be the region defined by $\ell_i \ge 0$ for each $i$ and $\sum_{i = 1}^r \ell_i \le r s$.   The number of $\ell_i$ is the number of lattice points in $T_{\log_2(\tilde \mu)}$.  Now if a lattice point lies in $T_s$,  then the unit cube with ``bottom corner'' at $s$ lies in $T_{s+1}$ and so the number of lattice points in $T_s$ is bounded by the volume of $T_{s+1}$.   Since $\tilde \mu \ge 2$, we have $\log_2(\tilde \mu) + 1 \le 2 \log_2(\tilde \mu)$,  and so the number of choices for $\ell$ is at most

$$ \Volume(T_{2 \log_2 \tilde \mu} ) = \frac{ ( 2 r \log_2 (\tilde \mu) )^r }{r!} \le C^r (\log \tilde \mu)^r. $$

We have $(\log H)^r$ choices for $a$ and then $C^r (\log \tilde \mu)^r$ choices of $\ell$ and then $\tilde \mu^r$ choices of $k$.
\end{proof} 

When $\mu(a,b)$ is large, we will use a crude upper bound for $\# (\Lambda^2 \cap \tilde K(a,b))$ coming from Lemma \ref{lemupper}:

$$ \# (\Lambda^2 \cap \tilde K(a,b)) \le \Cov_{1/10} (\tilde K(a,b)) \le C^d \Volume(\tilde K(a,b)). $$

As above, define $k_i$ so that $b_i = 2^{-k_i} \frac{H}{a_i}$.  Recall that $\mu = \mu(a,b) = \prod_i k_i$.  By the arithmetic-geometric mean inequality we have 

$$ \mu^{1/r} = (\prod_{i=1}^r k_i)^{1/r} \le \sum_{i=1}^r k_i / r. $$

And so

$$ - \sum_i k_i \le - r \mu^{1/r}. $$

So if $\mu(a,b) \ge \tilde \mu^r$,  

$$ \Volume(\tilde K(a,b)) \le C^d 2^{-\sum_i k_i} H^d \le C^d 2^{- r \tilde \mu} H^d. $$

Therefore, we have

 $$ \sum_{a,b : \tilde \mu^r \le \mu(a,b) \le (2 \tilde \mu)^r} \mu(a,b) \# (\Lambda^2 \cap \tilde K(a,b) ) \le C^r (\log H)^r \tilde \mu^r (\log \tilde \mu)^r  \cdot \tilde \mu^r \cdot C^d 2^{- \tilde \mu r} H^d = C^d \left[ \frac{C \tilde \mu^2 \log \tilde \mu}{2^{\tilde \mu}} \right]^r (\log H)^r H^d. $$


We now return to \eqref{frame}, which we rewrite as

 \begin{equation} \label{frame2} C^{-d} \frac{ H^d (\log H)^r }{\Delta_k} \le \sum_{\tilde \mu \textrm{ dyadic}} \left[  \sum_{a,b : \tilde \mu^r \le \mu(a,b) \le (2 \tilde \mu)^r} \mu(a,b) \# (\Lambda^2 \cap \tilde K(a,b) ) \right]. \end{equation} 
 
Recall that $r \ge d/2$.   So if $\tilde \mu$ is larger than $C \Delta_k^{\frac{1}{100d}}$, then $\left[ \frac{C \tilde \mu^2 \log \tilde \mu}{2^{\tilde \mu}} \right]^r $ is far smaller than $\frac{1}{C^d \Delta_k}$ and the contribution of that term on the right-hand side of \eqref{frame2} is negligible compared to the left-hand side.  After rearranging all these terms to the left hand side, what remains is

 \begin{equation} \label{frame3} C^{-d} \frac{ H^d (\log H)^r }{\Delta_k} \le \sum_{\tilde \mu \le C \Delta_k^{\frac{1}{100d}}} \left[  \sum_{a,b :  \mu(a,b) \le \tilde \mu^r} \mu(a,b) \# (\Lambda^2 \cap \tilde K(a,b) ) \right] \le C^r \Delta_k^{\frac{r}{100d}} \# (\Lambda^2 \cap R_H). \end{equation} 

Finally this gives \eqref{goal1}: 

$$\sum_{m \in O_{k, \le H}} N_{H_m} (O_k, H) =  \# (\Lambda^2 \cap R_H) \ge C^{-d} \frac{H^d (\log H)^r }{\Delta_k^{1.01}}. $$

The left-hand side in this equation includes the contribution of $m=0$, which is $2 |O_{k, \le H}|$.  By Lemma \ref{lemsizeOkH}, $2 |O_{k, \le H}| \le C^d H^d$.  If $H$ is large enough this contribution is dominated by the right-hand side.  We can check that taking $\log H  > C \Delta_k^{10/d}$ is sufficient.  With this condition, we have

$$\sum_{0 \not= m \in O_{k, \le H}} N_{H_m} (O_k, H)  \ge C^{-d} \frac{H^d (\log H)^r }{\Delta_k^{1.01}}. $$

Now since $|O_{k, \le H} | \le C^d H^d$, we get the desired lower bound on the average, \eqref{goal2}.  

\end{proof}

The proof of Theorem \ref{thmlowercir} is almost the same.  If we replace $H_m$ with $S_m$, then the geometry of the region $R_H$ changes.  The new definition of $R_H$ would be

$$ R_H = \{ (x,y) \in V^2: |x_i|, |y_i|, |x_i^2 + y_i^2 | \le H \} = \prod_{i=1}^r \{(x_i, y_i):   |x_i|, |y_i|, |x_i^2 + y_i^2 | \le H \}$$

Recall that if $i=1, ..., r_1$, then $x_i, y_i \in \RR$, and so the region $ \{(x_i, y_i):   |x_i|, |y_i|, |x_i^2 + y_i^2 | \le H \}$ is just a disk.  On the other hand, if $r_1 + 1 \le i \le r_1 + r_2 = r$, then $x_i, y_i \in \CC$, and the region $ \{(x_i, y_i):   |x_i|, |y_i|, |x_i^2 + y_i^2 | \le H \}$ is quite similar to the corresponding region in the analysis of $H_m$.  This is because we can factor $x_i^2 + y_i^2 = (x_i + i y_i) (x_i - i y_i)$ (sorry for the double use of the symbol $i$).  Now if $a = (a_{r_1 + 1}, ..., a_r)$, we can define $K(a)$ as the set of $(x,y) \in V^2$ so that

$$ |x_i|^2 + |y_i|^2 \le H \textrm{ if } 1 \le i \le r_1, $$

$$ | x_i + i y_i| \le a_i, |x_i - i y_i| \le \frac{H}{a_i} \textrm{ if } r_1 + 1 \le i \le r. $$

As before, $K(a) $ is a symmetric convex subset of $R_H$.  The number of choices for $a$ is now only $(\log H)^{r_2}$ instead of $(\log H)^{r_1 + r_2}$.  The rest of the analysis is similar.  In the proof of Theorem \ref{thmlowerhyp}, we used that $r_1 + r_2 \ge d/2$, which always holds.  As a substitute for this, we need to use the hypothesis that $r_2 \ge d/10$ that we added in Theorem \ref{thmlowercir}.  

We can use Theorem \ref{thmlowerhyp} to prove Theorem \ref{counterhyp} and Theorem \ref{thmlowercir} to prove Theorem \ref{countercircle}.   To prove Theorem \ref{counterhyp},  we let $H$ be a large constant to be chosen later.   We let $k$ be a sequence of totally real number fields of degree $d \rightarrow + \infty$ with $\Delta_k < C_{GS}^d$.   On the one hand, we have

$$[c_1 H]^d \le |O_{k, \le H} | \le \left[ C H \right]^d . $$

\noindent We will choose $H >  C_{GS}^{10} > C \Delta_k^{10/d}$, which is large enough to apply Theorem \ref{thmlowerhyp}.  By Theorem \ref{thmlowerhyp},  we can choose a non-zero $m \in O_{k, \le H}$ so that 

$$N_{H_m}(O_k, H) \ge C^{-d} \frac{ (\log H)^d}{\Delta_k^{1.01}} \ge \left[ C^{-1} C_{GS}^{-1.01} (\log H) \right]^d. $$

\noindent    We fix $H$ large compared to the other constants and we see that as $d \rightarrow + \infty$, $|O_{k,H}| \rightarrow + \infty$ and there is some $\beta > 0$ so that $|N_{H_m}(O_k, H)| \ge |O_{k, \le H}|^\beta$ for all choices of $k$.  This gives Theorem \ref{counterhyp}.

Finally,  we use Theorem \ref{thmlowercir} to give a counterexample to the unit distance conjecture, proving Theorem \ref{counterunit}.   By Theorem \ref{thmmart},  we can choose a sequence of number fields $k$ with degree $d \rightarrow \infty$ with $r_1 = r_2$ and with  $\Delta_k < C_{GS}^d$.   Let $H$ be a large constant to be chosen later.   In particular, we will choose $H > C_{GS}^{20} > C \Delta_k^{20/d}$ so that we meet the hypotheses of Theorem \ref{thmlowercir}.  By Theorem \ref{thmlowercir}, we can choose a non-zero $m \in O_{k, \le H}$ so that 

$$ N_{S_m}(O_k, H) \ge C^{-d} (\log H)^{r_2} = C^{-d} (\log H)^{d/3}. $$

Fix a real embedding $k \rightarrow \RR$,  and then set $P = O_{k,  \le 2 H}^2 \subset \RR^2$.   Recall from Lemma \ref{lemsizeOkH} that 

\begin{equation} \label{cardOkH}  2^{-d} C_{GS}^{-d} H^d \le 2^{-d} \frac{H^d}{\Delta_k^{1/2}} \le  |O_{k, \le H}| \le C_1 H^d
\end{equation}

We will choose $H$ large compared to the other constants,  so that as $d \rightarrow \infty$,  $|O_{k, \le H}| \rightarrow \infty$ and hence $|P| \rightarrow \infty$.  

We will count the number of pairs $p, q \in P$ with $|p - q|^2 = m$.   So

$$U(P) = \# \{ (p_1, p_2, q_1, q_2) \in O_{k, \le 2H}^4: (p_1 - q_1)^2 + (p_2 - q_2)^2 = m \}. $$

If $(p_1, p_2) \in O_{k, \le H}^2$,  then the number of choices for $(q_1,q_2)$ above is at least $N_{S_m}(O_k, H)$.   Therefore,  we have

$$ |U(P)| \ge |O_{k, \le H}|^2 N_{S_m}(O_k, H) \ge |O_{k, \le H}|^2 C^{-d} (\log H)^{d/3} \ge \tilde C^{-d} H^{2d} (\log H)^d. $$

On the other hand, 

$$ |P| = |O_{k, \le 2H}|^2 \le \tilde C^d H^{2d}. $$

By choosing $H$ a fixed large constant,  we see that

$$ U(P) \ge |P|^{1 + \epsilon_0}, $$

for a fixed $\epsilon_0 > 0$ as $d \rightarrow \infty$ and so $|P| \rightarrow \infty$.

\section{Upper bounds} \label{secupperbounds}

In the 1980s,  Spencer,  Szemeredi, and Trotter \cite{SST} proved that the number of distances determined by $n$ points is $O(n^{4/3})$.    That upper bound has not been improved in spite of a lot of effort by many people.   In this section,  we describe some of the methods used and discuss why it is difficult to improve them.

We will focus on the followng special case.   Divide the unit square $[0,1]^2$ into $n$ smaller squares of side length $n^{-1/2}$.   Suppose that $P$ consists of one point in each smaller square.   We call such a set uniformly spaced.   Assuming that $P$ is uniformly spaced,  we will sketch two different one page proofs that $U(P) = O(n^{4/3})$.   Even for uniformly spaced $P$,  we have no idea how to improve the exponent $4/3$.

\subsection{Double counting} \label{subsecdoublecount}

The first upper bound in the unit distance problem is based on double counting.   If we fix two points,  there are at most two unit circles that contain both of them.   We can exploit this fact to bound the incidences between the points and the unit circles.

 Suppose that $P$ is a set of points and $\Gamma$ is a set of unit circles.   Recall that $I(P, \Gamma) = \sum_{\gamma \in \Gamma} |\gamma \cap P|$.   If $\Gamma$ is the set of unit circles centered at $P$,  then $I(P,\Gamma)$ is the number of unit distances.
 
 Since any two points lie in at most two unit circles,  we have
 
 $$ \sum_{\gamma \in \Gamma} { |\gamma \cap P| \choose 2} \le 2 { |P| \choose 2} \sim |P|^2. $$
 
 If $|\gamma \cap P| \ge 2$,  then $|\gamma \cap P|^2 \lesssim { |\gamma \cap P| \choose 2}$.   Therefore,  we get
 
 $$ \sum_{\gamma \in \Gamma}  |\gamma \cap P|^2 \le |\Gamma| + C |P|^2.  $$
 
 Now by Cauchy-Schwarz,  we have
 
 \begin{equation} \label{doubcount} I(P, \Gamma) = \sum_{\gamma \in \Gamma} |\gamma \cap P| \le |\Gamma|^{1/2} \left( \sum_{\gamma \in \Gamma} |\gamma \cap P|^2 \right)^{1/2} \lesssim | \Gamma|  +  |\Gamma|^{1/2} |P|.  \end{equation}
 
 If $|\Gamma| = |P| = n$,  we get $I(P,\Gamma) \le c n^{3/2}$ and so the number of unit distances is $O(n^{3/2})$.   This argument appeared in \cite{Erd} in the 1940s.
 
 There is a barrier to improving the bound coming from finite field examples.   Unit circles also make sense over finite fields $\FF_q$ -- the unit circle with center $(x_0, y_0)$ is the set $\{ (x,y) \in \FF_q^2: (x-x_0)^2 + (y-y_0)^2 = 1 \}$.   Over finite fields, it is still true that any two points are contained in at most two unit circles.  So the argument above still shows that $n$ points determine at most $C n^{3/2}$ unit distances.   This upper bound is sharp when $P = \FF_q^2$,  because we have $q^2$ points and we have $\sim q^3$ unit distances because each unit circle $\gamma$ contains $\sim q$ points.   

Improving the bound $n^{3/2}$ requires us to distinguish unit circles in $\RR^2$ from unit circles in finite fields.   There are several different proofs.   All the proofs use topology in some way.

The proofs become easier in the uniformly spaced case.   The definition of uniformly spaced involves the structure of $\RR^2$ and doesn't have any analogue over finite fields.  

\subsection{Partitioning bound} \label{subsecpart}

In \cite{CEGSW},  Clarkson,  Edelsbrunner, Guibas, Sharir, and Welzl introduced a new approach to incidence geometry problems called the cutting method.  It is a divide and conquer method,  based on dividing space into cells.   With this method,  they gave elegant new proofs of previous results like the Szemeredi-Trotter theorem and the $n^{4/3}$ bound for the unit distance problem and also generalized the method to many more problems.

If the point set $P$ is well spaced,  then the cutting method is particularly clean and simple.   We divide $[0,1]^2$ into a $D \times D$ grid of smaller squares with side length $1/D$.   Here $D$ is  a free parameter that we will optimize later.  We call these smaller squares $O_i$,  and we let $P_i = P \cap O_i$.   Note that each unit circle enters $O(D)$ smaller squares $O_i$.   We let $\Gamma_i$ be the set of unit circles that enter $O_i$.

We have $I(P, \Gamma) = \sum_i I(P_i,  \Gamma_i)$.   We bound each term in the sum using the double counting bound \ref{doubcount} above:

$$ I(P, \Gamma) \lesssim \sum_i |\Gamma_i| + |\Gamma_i|^{1/2} |P_i|. $$

Since $|P|$ is uniformly spaced,  $ |P_i| \sim D^{-2} |P|$ for every $i$.   We also know that $\sum_i |\Gamma_i| \lesssim D |\Gamma$,  because each unit circle of $\Gamma$ enters $O(D)$ cells $O_i$.    Plugging in these facts,  we have

$$ I(P, \Gamma) \le D |\Gamma| + |P| D^{-2} \sum_i |\Gamma_i|^{1/2}. $$

Using Cauchy-Schwarz,  we can bound

$$ \sum_{i = 1}^{D^2} |\Gamma_i|^{1/2} \le D \left( \sum_{i} |\Gamma_i| \right)^{1/2} \le D^{3/2} |\Gamma|^{1/2}. $$

Plugging in we get

$$ I(P, \Gamma) \lesssim D |\Gamma| + D^{-1/2} |P| |\Gamma|^{1/2}. $$

Now we optimize the choice of $D$.   The optimal choice is $D = |P|^{2/3} |\Gamma|^{-1/3}$,  which makes the right-hand side equal to $|\Gamma|^{2/3} |P|^{2/3}$.   However, we are constrained to choose $D \ge 1$.   So we choose $D$ to be the maximum of 1 and $|P|^{2/3} |\Gamma|^{-1/3}$  which gives

$$ I(P, \Gamma) \lesssim |\Gamma| + |P|^{2/3} |\Gamma|^{2/3}. $$

If $|P| = |\Gamma| = n$,  we get the upper bound $I(P, \Gamma) = O(n^{4/3})$.  

We remark that this upper bound applies not just to unit circles but also to straight lines and also to many other classes of curves.   The bound applies whenever any two points lie in $O(1)$ curves of $\Gamma$ and any curve of $\Gamma$ intersects $O(D)$ squares $O_i$ above.   It is also straightforward to generalize the method to higher dimensions.   For straight lines,  this upper bound is actually sharp,  matching examples based on integer grids.   But for unit circles,  there is no known matching example.   Even in spite of the recent new examples,  I believe that the $n^{4/3}$ upper bound for unit circles can be improved.

If the points of $P$ are not uniformly spaced,  then \cite{CEGSW} found a remarkable way to modify the argument so that it still works.   Instead of cutting the square into standard $1/D \times 1/D$ subsquares,  they designed a partition of the plane adapted to the geometry of $P$ which accomplished the same thing.   Their partition is built by starting with $D$ random curves from $\Gamma$ and using them as ``cell walls''.   These cell wall curves divide the plane into $\sim D^2$ cells.   After possibly refining these cells,  we get a partition into $\sim D^2$ cells $O_i$.   The key features of this partition are that each cell contains approximately the same number of points of $P$ and that each curve of $\Gamma$ intersects $O(D)$ of the cells.   The rest of the argument is the same as above.

This argument uses topology in order to bound the number of cells that each curve $\gamma$ enters.   The topological argument goes like this.   For unit circles or straight lines,  each curve $\gamma \in \Gamma$ crosses any other $\gamma' \in \Gamma$ only $O(1)$ times.  Therefore,  each $\gamma$ crosses the cell walls $O(D)$ times.   And therefore each $\gamma$ enters $O(D)$ cells.

\subsection{Spectral bound}

If $P$ is uniformly spaced,  there is an alternate proof of the $O(n^{4/3})$ upper bound on unit distances based on a Fourier analytic or spectral argument developed by Iosevich-Rudnev \cite{IR} and Vinh \cite{V}.  

This argument is easiest to explain in finite fields, so let us start there.  Suppose $\FF_q$ is a finite field.  The analogue of the unit circle is the set $S := \{ (x,y) \in \FF_q^2: x^2 + y^2 = 1 \}$.  Suppose $P \subset \FF_q^2$ is a point set.  Then the number of unit distances in $P$ is

\begin{equation} \label{UPenFq}
 U(P) = \sum_{x, y \in \FF_q^2} 1_{P}(x) 1_P(y) 1_S(x-y). 
 \end{equation}

There is a natural linear operator in this formula.  If $f: \FF_q^2 \rightarrow \RR$, we define

$$ T_S f(x) := \sum_{y \in \FF_q^2} f(y) 1_S(x,y).$$

Then we can rewrite the number of unit distances in terms of the linear operator $A_S$:

\begin{equation} \label{UPAS}
 U(P) = \langle 1_P, T_S 1_P \rangle. 
 \end{equation}

Given this formula, it makes sense to try to study $U(P)$ using the spectral theory of the operator $T_S$ -- by looking at the singular values and singular vectors of $T_S$.  

The operator $T_S$ is actually a convolution operator: $T_S f = f * 1_S$.  Therefore, its singular values and singular vectors can be computed from the Fourier transform of $1_S$.  The Fourier transform is a non-trivial exponential sum which can be bounded using Weil's estimate for exponential sums cf. \cite{IR}.  The result is that $T_S$ has one singular value of size $\sim q$ with singular vector equal to the all 1's vector, $1_{\FF_q^2}$, and all other singular values of $T_S$ have norm $\lesssim q^{1/2}$.  Now we can decompose $1_P = \frac{|P|}{q^2} 1_{\FF_q^2} + \tilde 1_P$, where $\tilde 1_P$ is perpendicular to $1_{\FF_q}$.  Plugging into \eqref{UPAS} and expanding, we see that

\begin{equation} \label{specbound}
U(P) \lesssim q^{-1}|P|^2 + q^{1/2} |P|.
\end{equation}

In this formula, the first term represents the number of unit distances in a random set of size $|P|$, and so it cannot be improved.  Recall that the double counting bound gives $|U(P)| \lesssim |P|^{3/2}$.  The spectral bound improves on the double counting bound when $q \ll |P| \ll q^2$.  

This spectral method also makes sense in Euclidean space with a little work.  Let $1_{S^1}: \RR^2 \rightarrow \RR$ be the characteristic function of the unit circle.   Suppose $P \subset \RR^2$, and let $\mu_P$ be the measure consisting of a unit Dirac mass at each point $p \in P$.  Then we have the formula

\begin{equation} \label{UPeneuc}
 U(P) = \int \int \mu_P(x) \mu_P(y) 1_{S^1}(x-y). 
 \end{equation}

The shape of this formula appears in many problems in analysis / physics.   The general shape is an integral of the form $\int \int \mu(x) \mu(y) K(x-y)$.   For example,  if $\mu$ is a charge density,  and $K(x-y)$ represents the potential energy coming from the interaction between a point charge at $x$ and a point charge at $y$,  then the total potential energy would be $\int \int \mu(x) \mu(y) K(x-y)$.    There is a natural linear operator appearing in this formula.  We define $T_K f(x) = f * K (x) = \int f(y) K(x-y) dy$.  Then we have 

$$ \int \int \mu(x) \mu(y) K(x-y) = \langle \mu, T_K \mu \rangle, $$

\noindent and it makes sense to study this type of integral using the spectral theory of the operator $T_K$.  Because $T_K$ is a convolution operator, its spectral theory can be read off from the Fourier transform of $K$.

However, the formula \eqref{UPeneuc} is too singular to apply this method directly.  The function $1_{S^1}$ is not continuous. The method works well if $K$ is smooth and compactly supported.  If $K (x)$ is smooth, compactly supported, and $1_{S^1}(x) \le K(x)$, then we can bound 

\begin{equation} \label{UPeneuc2}
 U(P) = \int \int \mu_P(x) \mu_P(y) 1_{S^1}(x-y) \le \int \int \mu_P(x) \mu_P(y) K(x-y) = \langle \mu_P, T_K \mu_P \rangle, 
 \end{equation}

\noindent and we can bound this expression using the spectral theory of $T_K$, which is given by $\hat K$.  

We can choose $K$, and the best known choice is a smooth approximation of the characteristic function of an annulus.  We let $\Ann_{1, w}:= \{ x: 1-w \le x \le 1+w \}$ and we let $K_w$ be a smooth approximation of $1_{\Ann_{1,w}}$.  Here $w > 0$ is a parameter we can choose below.  With this choice of $K$,  we are counting distances that lie in $[1-w, 1+w]$, which gives an upper bound for the number of exact unit distances.  As we increase $w$, we are increasingly overcounting, but the function $K_w$ becomes smoother and so the spectral theory of the operator $T_{K_w}$ becomes nicer.

When we study the singular values of $\hat K$, we see some large singular values at low frequency and we need to know something about the spacing of $P$ to control them.  Recall that if $P \subset [0,1]^2$ is a set of $n$ points, we say $P$ is uniformly spaced if it contains $1$ point in each sub-square of side length $n^{-1/2}$.  When $P$ is uniformly spaced in this sense, we get a bound very similar to the one in the finite field case.  The parameter $w$ is analogous to $1/q$ and the bound is 

\begin{equation} \label{specbound2}
U(P) \lesssim w |P|^2 + w^{-1/2} |P|. 
\end{equation}

Since $w$ is a parameter, we can now optimize it.  The optimal choice is $w \sim |P|^{-2/3}$, which gives $|U(P)| \lesssim |P|^{4/3}$.  


\subsection{Why is the exponent 4/3 hard to improve?}

If $P$ is uniformly spaced, then we have sketched two simple proofs that $|U(P)| \lesssim |P|^{4/3}$.  No one knows how to improve the exponent 4/3, even for uniformly spaced sets.

There are several points of view about which help explain why the exponent 4/3 is hard to improve.

One issue is that 4/3 is actually the sharp exponent for two different cousin problems to the unit distance problem.  To set these up, let $S^1$ be the unit circle centered at 0, and recall that $U(P) = \# \{ p, q \in P : p-q \in S^1 \}$.   If $\gamma$ is another curve in $\RR^2$, let us define

$$ U_\gamma(P) = \# \{ p, q \in P : p - q \in \gamma \}. $$

\noindent If $\gamma$ is any strictly convex $C^2$ curve, then the partitioning method shows that $ |U_\gamma(P)| \lesssim |P|^{4/3}$.   The spectral method 
also works if $P$ is uniformly spaced and the curvature of $\gamma$ is roughly constant.  The bound $|P|^{4/3}$ is actually sharp for some choices of $\gamma$ in this class.

The first example is the parabola -- when $\gamma$ is the curve $y = x^2$.

The second example is much rougher.     In the 1920s, Jarnik constructed a strictly convex curve $\gamma$ that contains $\sim H^{2/3}$ integer points in the square $[-H,H]^2$.  If we take $P = \ZZ_{\le 2H}^2$,   we see that $|P| \sim H^2$ and $U_\gamma(P) \sim H^2 (H^{2/3}) \sim |P|^{4/3}$.  The Jarnik curve $\gamma$ is not unique -- there are many different such curves.  

To improve the exponent 4/3 for the unit circle we have to distinguish it from both the parabola and the Jarnik curves.

The spectral method applies to a wide range of problems and there are a number of examples where it is hard to improve.  In particular, the spectral method we used here with an auxiliary function $K(x) \ge 1_{S^1}(x)$ is similar to the spectral approach to the density of sphere packings and error-correcting codes introduced by Delsarte in the 1970s.  There are a number of variants of that problem where we don't know how to improve on the spectral method -- for instance the asymptotics of the optimal density of sphere packings in $\RR^d$ as $d \rightarrow \infty$.

For another perspective,  the unit distance problem is a special case of the following general problem studied in graph theory and computer science.  We are given a graph $G$ with vertex set $V$ and edge set $E$.  If $P \subset V$, we write $E(P)$ for the number of edges with both endpoints in $P$.   For a given graph $G$, the problem is to estimate $E_G(n) = \max_{|P| =n } E(P)$.  The unit distance problem in $\FF_q^2$ is exactly this problem where $G$ is a graph with vertex set $V = \FF_q^2$ and edge set $E = \{ (p,q): p - q \in S^1 \}$.  This graph has $q^2$ vertices and each vertex has degree $\sim q$.  

 The spectral method is a standard method in graph theory for bounding $E_G(n)$ in terms of the spectrum of the adjacency matrix of $G$.  For instance, if $G$ is a random graph with degree $D$ and with $V$ vertices, then with high probability, the adjacency matrix of $G$ has one singular value of size $D$ (with singular vector equal to the all 1's vector) and other singular values of size $\lesssim D^{1/2}$.   This leads to the bound

\begin{equation} \label{specbound3}
E(P) \lesssim \frac{D}{V}  |P|^2 + D^{1/2} |P|. 
\end{equation}

\noindent These bounds exactly correspond to the ones for the finite field unit distance problem in \eqref{specbound}.

If $G$ is a random graph, then with high probability, the true size of $E_G(n)$ is much smaller than the spectral bound \eqref{specbound3}.  The spectral bound can be computed in polynomial time, giving an efficient proof that $E_G(n) \lesssim \frac{D}{V} n^2 + D^{1/2} n$.   In the regime $V^{1/2} \le n \le V$,  to the best of my knowledge,  the spectral bound is the strongest bound that we know how to prove efficiently.  Of course in exponential time, we can compute $E_G(n)$ exactly.  But if $G$ is a random graph and $V^{1/2}  \ll n \ll V$, then I don't know how to give a polynomial length proof that $E_G(n) \ll  \frac{D}{V} n^2 + D^{1/2} n$.  It even seems plausible that no such short proofs exist.  The issues here are closely related to a number of well studied problems in average case computational complexity in computer science, such as the planted clique problem, cf.  \cite{BHKKMP}.

It is conceivable that there is a better way to bound $E_G(n)$ for general graphs, which would be a remarkable discovery in computer science and graph theory.   But without such a discovery,  one should look for ways to bound $U(P)$ that are based on special structures in the unit distance problem that don't appear when bounding $E_G(n)$ for a general graph $G$ like a random graph.

\subsection{Algebraic geometry point of view}

One special feature about the unit distance problem is that we can describe it using polynomial equations.  This structure is used crucially in computational work that evaluates $U_{max}(n)$ for small values of $n$, such as the papers \cite{Ale} and \cite{Eng}.

Suppose we would like to compute $U_{max}(n)$ for a particular $n$.  We are considering $n$ points $(x_i, y_i) \in \RR^2$.  Next we select a set $U$ of pairs $(i,j) \in \{1, ..., n\}^2$ with $i < j$.  We have to decide whether there are distinct points $(x_i, y_i) \in \RR^2$ so that for all $(i,j) \in U$, $(x_i - x_j)^2 + (y_i - y_j)^2 = 1$.  This is a system of polynomial equations with $2n$ real variables, $U$ equations, and a few non-equalities to make sure the points $(x_i, y_i)$ are distinct.  Tarski-Seidenberg proved a remarkable result in logic / algebraic geometry which implies that there is an algorithm to decide any such problem.

\begin{theorem} \label{thmtarseid} If $p_j(x_1, ..., x_n)$ are polynomials with rational coefficients, and $S$ is a Boolean formula based on the statements $p_j(x_1, ..., x_n) = 0$ and/or $p_j(x_1, .., x_n) > 0$, then there is an algorithm which decides whether there is any point $(x_1, ..., x_n) \in \RR^n$ that obeys the Boolean formula $S$.  \end{theorem}

\noindent Because of Theorem \ref{thmtarseid}, there is an algorithm that computes $U_{max}(n)$ for any $n$.  For each set of pairs $U$ in $\{1, ..., n \}^2$, we can use this algorithm to decide whether there is a set of $n$ distinct points $(x_i, y_i) \in \RR^2$ so that each pair $(i,j) \in U$ is a unit distance.  Then among all $U$ that correspond to unit distances, we take the maximum cardinality.
It was not at all obvious that $U_{max}(n)$ can be computed at all, since it is defined as the maximum over an infinite set.  Theorem \ref{thmtarseid} gives a way to compute that maximum, which is a significant insight.

The best running time of the algorithm in Theorem \ref{thmtarseid} is an important problem at the interface of computer science and algebraic geometry.   See Basu's survey article \cite{Ba} for an engaging overview of this subject.   The original work of Tarski-Seidenberg produces an algorithm with incredibly slow running time.   A number of people contributed important ideas to improve the running time,  including Collins,  Basu, Pollack, and Roy, leading up to algorithms that run in exponential time.   See \cite{Ba} for more details.  

The papers \cite{Ale} and \cite{Eng} indeed use this approach to help compute $U_{max}(n)$.   The computation becomes difficult for large $n$ because the number of choices for $U$ grows exponentially,  and the running time to evaluate each $U$ also grows exponentially.   Papers directly using these algorithms got stuck around $n=15$ and the paper \cite{Ale} introduced some additional tricks to go further, getting up to $n=21$.

There is a good reason to think that the running time of the algorithm in Theorem \ref{thmtarseid}
should be at least exponential in $n$: this class of problems includes difficult discrete problems like 3SAT.  We can encode an example of 3SAT into a Tarski-Seidenberg problem as follows.  First we add polynomial equations  $x_i^2 - x_i = 0$ which force each $x_i$ to be 0 or 1.  Next we consider a 3SAT clause, such as 

`` $x_3 = 0$ or $x_7 = 1$ or $x_{13} = 0$. ''

We can encode this clause as a polynomial equation:

$$ x_3 (x_7 - 1) x_{13} = 0. $$

\noindent  In computer science, it is widely believed that deciding 3SAT requires exponential time, and so solving the Tarski-Seidenberg decision problem must also take exponential time.    



To summarize, the unit distance problem can be converted into a sequence of Tarski-Seidenberg decision problems, which shows that $U_{max}(n)$ is computable.  But I don't see how to get any more mileage out of the mere fact that the unit distance problem can be written in this way.  Known bounds for the unit distance problem are based on special features of the particular equations that arise.

\subsection{Struggling with the obstacles}

Mathematicians in incidence geometry have been struggling with these obstacles for decades.  For a number of problems in the field, mathematicians have proven bounds which improve on the partitioning method and the spectral method, sometimes distinguishing between similar sounding objects like the circle and the parabola.  We mention here a few interesting examples of such work.

In the finite field setting,  Bourgain-Katz-Tao \cite{BKT} proved and used sum-product estimates over finite fields to get bounds for point-line incidences that do not follow from double counting or spectral methods.   For example,  in $\FF_q^2$,  consider a set $P \subset \FF_q^2$ with $|P| = q$ and a set of lines $L$ with $|L| = q$.     Double counting and spectral methods both give the upper bound $I(P, L) \lesssim q^{3/2}$.   This upper bound is actually sharp if $q = p^2$ for a prime $p$ so that $\FF_q$ has a subfield $\FF_p$ of index 2.   In this case,  we can choose $P = \FF_p^2 \subset \FF_q^2$ and we can choose $L$ to be the set of all lines of the form $y = mx + b$ with $b, m \in \FF_p$.   Each such line contains $p$ points of $P = \FF_p^2$,  and so $I(P, L) = p^3 = q^{3/2}$.    But if $q$ is prime,  and $|P| = |L| = q$,  then \cite{BKT} proved that $I(P, L) \lesssim q^{3/2 - \epsilon_0}$ for some $\epsilon_0 > 0$.   The double counting bound and the spectral bound do not distinguish between prime and non-prime $q$,  and so quite different ideas are needed.  

One key obstacle in the unit distance problem is that unit circles behave differently from unit parabolas (translates of the curve $y = x^2$).   The methods presented above do not distinguish these two types of curves.   Elekes found a simpler cousin problem where unit circles and unit parabolas behave differently,  and he proved bounds that distinguish them.   We will discuss this work more in Section \ref{secgoals}.

For the unit distance problem,  the bounds coming from the partitioning method are the best current bounds  even if we assume that $P$ is uniformly spaced.   But there are many other natural incidence problems in the plane,  including the case of circles,  the case of ellipses,  the case of degree $d$ algebraic curves,  ...   The partitioning method and the spectral method have analogues in all these cases.   The method of lenses has led to further improvements on many of these problems.

The method of lenses was first developed by Tamaki and Tokuyama \cite{TT} and Aronov-Sharir \cite{AS} for the case of circles.   Later,  using different methods,  Aronov, Sharir, and Zahl \cite{SZ} developed a more general version,  which applied to all algebraic curves.    Using these results on lenses in an indirect way,  Zahl \cite{Z} improved the bounds for the unit distance problem in $\RR^3$.   

If $P$ is a set of $n$ points in $\RR^3$,  an integer grid construction gives $U(P) \sim n^{4/3}$.   The partitioning method gives an upper bound $U(P) \lesssim n^{3/2}$.   If the point set $P$ is uniformly spaced,  the spectral method also gives $U(P) \lesssim n^{3/2}$.    For any $P \subset \RR^3$, Zahl proved that $U(P) \lesssim n^{3/2 - \epsilon_0}$ for some $\epsilon_0 > 0$.    All the obstacles we mentioned above appear in the 3-dimensional problem just as in the 2-dimensional problem,  and the proof overcomes all of them.  There is no straightforward way to adapt this proof from 3 dimensions to 2 dimensions, but it might be worth searching more for a clever way of applying these ideas.

In another direction,  Pach,  Raz, and Solymosi have an interesting program for improving the upper bounds in the unit distance problem, cf. \cite{PRS}.

To summarize, topological methods like the partitioning method led to the upper bound $U_{max}(n) \lesssim n^{4/3}$ in the early 80s.  In the uniformly spaced case, there are two different one page proofs of this bound.  Mathematicians in the field have introduced a bunch of clever ideas which allow to improve on the partitioning method for a range of cousin problems, but they have not led to any improvement in the original unit distance problem, even in the uniformly spaced case.  

The upper bound $U_{max}(n) \lesssim n^{4/3}$ applies not just to the unit distance problem but to a broad range of problems.  For a few of these problems, the upper bound $n^{4/3}$ is tight, but for the vast majority of them, it can probably be improved.  Improving the exponent 4/3 for any one explicit problem of this type would be a significant milestone.

\section{Classification and structural results} \label{secgoals}

One set of problems in incidence geometry asks to estimate $I_{max, C}(m,n)$ as accurately as possible for different classes of objects $C$.   For instance, we could ask to estimate $U_{max}(n)$ as accurately as possible, perhaps trying to find the infimal exponent $\gamma$ so that $U_{max}(n) \lesssim n^\gamma$.  
Until recently it was generally believed that this infimal exponent $\gamma$ was 1.  We now know that $\gamma > 1$ and that $\gamma \in [1.014, 4/3]$.  

It is not clear to me that we should expect any nice formula for this optimal exponent $\gamma$.  Suppose for a moment that the optimal exponent comes from a version of the recent constructions, taking $P$ to be $O_{k, \le H}^2$ for a well-chosen sequence of number fields $k$ and numbers $H$.  The recent constructions are not easy to optimize.  For instance, they all involve the constant in the bound $\Delta_k < C_{GS}^d$.  The best constant $C_{GS}$ has been investigated extensively in the number theory literature.  We have upper and lower bounds that match up to a factor of around 10.  But there does not seem to be any plausible conjecture for the exact value of $C_{GS}$.  In the constructions we have considered, the bound for $\gamma$ that we get depends on the value of $C_{GS}$.  Optimizing these constructions looks at least as messy and complicated as optimizing $C_{GS}$.

There remain other problems where it looks plausible that the best exponents have a clean form.  For instance, if we consider incidences of points with circles or with parabolas or with ellipses or with degree $d$ algebraic curves for fixed $d$, then the best known example comes from a construction using integer points of bounded sizes on the corresponding variety.  We can replace the integers by $O_k$ in these constructions, but the exponents do not seem to depend on the choice of $k$.  As another example, we can replace the unit circle by a higher degree curve, such as the curve $\gamma_3$ defined by $x^3 + y^3 = 1$.  For the corresponding version of the unit distance problem, number theoretic constructions are not known to give interesting examples, and so the best known examples are based on symmetry and have $U_{\gamma_3}(P) \sim |P|  \log |P|$.  It looks plausible that $U_{\gamma_3, max}(n) = n^{1 + o(1)}$.  

Now let's come back to the question whether the best exponent in the unit distance problem comes from an arithmetic construction of the form $O_{k, \le H}^2$.  I think that this question is really more interesting than the exact value of the best exponent.  Arithmetic constructions give the best known examples for a broad range of problem in incidence geometry.  In fact, for a broad range of problems, all known highly overdetermined examples come from arithmetic constructions.  I think it is a fundamental question whether all such examples have to come from arithmetic constructions.  

This question fits into a longstanding theme in incidence geometry: trying to find classifications and structural results for configurations with many incidences.  

Suppose we consider an incidence problem involving $n$ points and curves.  The problem has about $n$ free parameters.  Each incidence corresponds to an equation.  If the number of incidences is much more than the number of parameters, the problem is called overdetermined.  Intuitively, overdetermined examples should only exist because of some special structure in the problem that explains them.  Mathematicians have worked hard on trying to understand and classify these special structures.  Here are two examples of work in this direction.

\subsection{Abelian group structure}

Elekes, Ronyai, and Szabo developed some of the first structural results in incidence geometry.  They considered incidence problems involving three sets $A, B, C$ instead of two sets $A,B$.  Given a 2-dimensional variety $Z \subset \RR^3$, they asked about the maximum size of $Z \cap (A \times B \times C)$, where $A, B, C \subset \RR$ with $|A| = |B| = |C| = n$.  For varieties $Z$ that they considered, given $a,b \in \RR$, there are $O(1)$ choices of $c \in \RR$ so that $(a,b,c) \in Z$, and so $| Z \cap (A \times B \times C) | \lesssim n^2$.  This upper bound is sometimes sharp.  For instance, if $Z$ is the variety defined by $x_1 + x_2 + x_3 = 0$, and if $A = B = C = \ZZ_{\le n}$, then $|Z \cap (A \times B \times C) | \sim n^2$.  Next suppose that $Z$ is the variety $x_1^2 + x_2^2 + x_3^2 = 1$.  Here the upper bound $n^2$ is still sharp because of the example $A = B = C = \{ \sqrt{j} \}_{j \in \{1, .., n \}}$.  But Elekes and Szabo proved that for most $Z \subset \RR^3$, there is a stronger upper bound for $|Z \times (A \times B \times C)|$.  For example, if $Z$ is defined by $x_1 x_2 + x_1 x_3  + x_2 x_3$, then Elekes-Szabo proved a stronger upper bound for $|Z \cap (A \times B \times C)|$.  In the first two examples, the variety $Z$ is closely related to an abelian group structure, possibly after a change of coordinates.  Without that structure, there are stronger upper bounds.

\begin{theorem} \label{thmes} (Elekes-Szabo, vague statement) There is a constant $\eta > 0$ so that the following holds.  Suppose that $Z \subset \RR^3$ is a subvariety (or smooth submanifold??) and $A, B, C \subset \RR$ with $|A| = |B| = |C| =n$.  Then one of the following two options holds: 

\begin{itemize}

\item The variety $Z$ encodes an abelian group structure in a precise sense.

\item $|Z \cap (A \times B \times C) | \lesssim n^{2 - \eta}$.

\end{itemize}

\end{theorem}

Elekes gave a striking application of this result to a problem about unit circles.  Suppose that $p_1, p_2, p_3 \in \RR^2$.  Let $\Gamma_i$ be a set of curves passing through $p_i$.  We define the number of triple points determined by $\Gamma_1, \Gamma_2, \Gamma_3$ by

\begin{equation} \label{deftrip}
T(\Gamma_1, \Gamma_2, \Gamma_3) = \# \{ (\gamma_1, \gamma_2, \gamma_3) \in \Gamma_1 \times \Gamma_2 \times \Gamma_3: \gamma_1 \cap \gamma_2 \cap \gamma_3 \not= \emptyset \}. 
\end{equation}

\noindent We suppose that $|\Gamma_i| = n$ for $i = 1,2,3$ and try to maximize $T$.  We consider three different classes of curves: unit circles, straight lines, and unit parabolas (parabolas of the form $y = x^2 + a x + b$).  For straight lines and for unit parabolas, it can happen that $T(\Gamma_1, \Gamma_2, \Gamma_3) \sim n^2$.  But for unit circles, Elekes-Szabo proved that 
$T(\Gamma_1, \Gamma_2, \Gamma_3) \lesssim n^{2 - \eta}$.  The set of straight lines through $p_i$ form a 1-parameter family, which we can parametrize by $\theta_i \in \RR$.  The condition that $\gamma_1(\theta_1) \cap \gamma_2(\theta_2) \cap \gamma_3(\theta_3) \not= \emptyset$ is a codimension 1 condition on $\theta_1, \theta_2, \theta_3$ -- it corresponds to a 2-dimensional submanifold $Z_{line} \subset \RR^3$.  Applying the same approach to unit circles or unit parabolas, we get two more submanifolds $Z_{unit circle}$ and $Z_{unit parabola}$.  For lines and unit parabolas, the submanifold $Z$ encodes an abelian group structure, but for unit circles, Elekes-Szabo proved that it does not.

 If $|A| = |B| = |C| = n$, then we have $3n$ real variables, and so a configuration is overdetermined if $|Z \cap (A \times B \times C)| > 3n$.  The value of $\eta$ in Theorem \ref{thmes} was improved in ... and \cite{SZ}, and the best current value is $\eta = 2/7$.  So Theorem \ref{thmes} gives a detailed structural description of the scenario when $|Z \cap (A \times B \times C)| \gg n^{1 + \frac{5}{7}}$.  
Makhul,  Roche-Newton, Warren, and de Zeeuw \cite{MRWZ} gave an example where $Z$ does not have abelian group structure and yet $|Z \cap (A \times B \times C)| \sim n^{3/2}$.  It is conceivable that Theorem \ref{thmes} holds with $\eta = 1/2$.  If $|Z \cap (A \times B \times C)| \sim n^\alpha$ for $1 <\alpha \le 3/2$, the problem is still highly overdetermined and so intuitively such configurations should all have some type of special structure, but there is not currently any conjectural classification of the special structures that may occur.

 \subsection{Low degree algebraic structure}
 
 The polynomial method, developed in \cite{D}, \cite{GK1}, \cite{GK2}, and other papers, is a tool for finding low degree polynomial structure in incidence geometry problems.  Here is an example.  Suppose that $\frak L$ is a set of $n$ lines in $\RR^3$.  Let $P_2(\frak L)$ be the set of points in $\RR^3$ that lie in at least 2 lines in $\RR^3$.  Since two lines intersect in at most one point, $|P_2(\frak L)| \le {n \choose 2}$.  It can happen that $|P_2(\frak L)| = {n \choose 2}$ if all the lines of $\frak L$ lie in a plane.  More generally, it can happen that $|P_2(\frak L)| \sim n^2$ if all the lines of $\frak L$ lie in a degree 2 algebraic surface, such as $z = xy$.  The surface $z= xy$ contains two families of lines: the lines $\ell_{1,a}$ parametrize by $t \mapsto (a, t, at)$ and the lines $\ell_{2,b}$ parametrized by $t \mapsto (t, b, bt)$.  For any $a,b$, $\ell_{1,a} \cap \ell_{2,b} \not= \emptyset$.  Choosing $n/2$ lines of the first family and $n/2$ lines of the second family, we get $|P_2(\frak L)| = n^2/4$.  In \cite{GK2}, Nets Katz and I proved that if $|P_2(\frak L)| \gg n^{3/2}$, then the lines of $\frak L$ must cluster into planes or degree 2 surfaces.
 
 \begin{theorem} \label{thmgk} (\cite{GK2}) If $\frak L$ is a set of $n$ lines in $\RR^3$  and $|P_2(\frak L)| \gg n^{3/2}$, then there is a plane or degree 2 surface that contains $\gtrsim \frac{ |P_2(\frak L)| }{n}$ lines of $\frak L$.
  \end{theorem}

The polynomial method has an application to a problem about distances called the distinct distances problem.  In the paper \cite{Erd}, Erdos posed several  problems,  including the distinct distances problem and the unit distance problem.  In \cite{GK2}, Nets Katz and I solved the distinct distances problem in dimension 2.  The solution combines a framework introduced by Elekes-Sharir with incidence estimates for lines in $\RR^3$ such as Theorem \ref{thmgk}.
 
Theorem \ref{thmgk} gives information about the special structures that must appear when $|P_2(\frak L)|$ is large.  A line in $\RR^3$ is determined by 4 real parameters, and so if $\frak L$ is a set of $n$ lines, then there are $4n$ real parameters.  The condition that two lines in $\RR^3$ intersect is a codimension 1 condition.  Therefore, if $|P_2(\frak L)| > 4n$, then the configuration is overdetermined.  If $|P_2(\frak L)| \gg n^{3/2}$, then Theorem \ref{thmgk} leads to a full classification of possible configurations.  All such configurations are based on lines clustering into planes or degree 2 surfaces.  In the regime, $n \ll |P_2(\frak L)| \lesssim n^{3/2}$, the problem is still highly overdetermined but we are not able to prove any structural information about $\frak L$.

There are many open problems in this area.  One significant open problem is to generalize Theorem \ref{thmgk} to higher dimensions.  This problem connects to subtle questions in algebraic geometry, and it is a place where I think the field could benefit from collaboration between people in combinatorics and people in algebraic geometry.  See \cite{Gu} for more background about this problem and related questions in algebraic geometry (and feel free to email me about it).

\subsection{Arithmetic structure?}

Many people have wondered about a classification theorem for near-optimal configurations in the point-line incidence problem.  If $|P| = |L| = n$ and $I(P,L) > 6n$, then the configuration is overdetermined.  We have a few examples with $I(P,L) = n^\alpha$ with $\alpha > 1$, but they are all in some sense based on studying integer points on varieties (perhaps over a number field).  On the other hand, even assuming $I(P,L) \sim n^{4/3}$, no one has managed to prove that the configuration has any kind of arithmetic structure.

More broadly, arithmetic structure appears in the best known examples for a wide variety of problems in incidence geometry.  
These examples based on number fields and arithmetic seem to hint at an underlying connection between incidence geometry and number theory.  I think such a connection would be interesting, because the questions in incidence geometry all take place over the real numbers and don't mention integers or number fields.  On the other hand, it may be that there are other really different examples in incidence geometry which are not based on arithmetic.   Understanding whether there is such an underlying connection is a longstanding philosophical question in the field.  If there is such a connection, it would be interesting to learn where the connection comes from and how it works.  And if there are non-arithmetic examples, it would give a new direction of examples and phenomena to explore.

\end{document}